\documentclass[reqno,11pt,a4paper]{amsart}

\usepackage{amsmath}
  \usepackage{paralist}
  \usepackage{graphics} 
  \usepackage{epsfig} 
\usepackage{graphicx}  \usepackage{epstopdf}
 \usepackage[colorlinks=true]{hyperref}
\hypersetup{urlcolor=blue, citecolor=red}
\usepackage{amssymb, array, amsmath, amscd, mathtools}
\usepackage{graphicx,epstopdf} 
\usepackage[caption=false]{subfig} 
\usepackage[utf8]{inputenc}
\usepackage{amssymb}
\usepackage{bbm}
\usepackage{overpic}
\usepackage{xcolor}
\usepackage{ifthen}
\usepackage{cite}

\theoremstyle{plain}
  
\newtheorem{theorem}{Theorem}[section]

\newtheorem{lemma}[theorem]{Lemma}
\newtheorem{cor}[theorem]{Corollary}

\theoremstyle{remark}
\newtheorem{remark}{Remark}[section]

\theoremstyle{definition}
\newtheorem{definition}{Definition}[section]

\newcommand{\linspan}{\mathop{\rm span}\nolimits}

\newcommand{\diver}{\mathop{\rm div}\nolimits}
\newcommand{\curl}{\mathop{\rm curl}\nolimits}

\newcommand{\sign}{\mathop{\rm sign}\nolimits}

\newcommand{\rest}{\left.\kern-2\nulldelimiterspace\right|_}
\newcommand{\norm}[2]{|#1|_{#2}}

\newcommand{\p}{\partial}

\newcommand{\e}{\varepsilon}
\newcommand{\ed}{\mathrm d}

\newcommand{\N}{{\mathbb N}}
\newcommand{\R}{{\mathbb R}}

\newcommand{\D}{{\mathrm D}}

\newcommand{\nnn}{\mathbf n}

\newcommand{\BB}{{\mathcal B}}
\newcommand{\CC}{{\mathcal C}}

\newcommand{\FF}{{\mathcal F}}
\newcommand{\GG}{{\mathcal G}}

\newcommand{\KK}{{\mathcal K}}
\newcommand{\LL}{{\mathcal L}}

\newcommand{\RR}{{\mathcal R}}
\newcommand{\sS}{{\mathcal S}}

\newcommand{\ZZ}{{\mathcal Z}}

\newcommand*{\textlabel}[2]{%
  \edef\@currentlabel{#1}
  \phantomsection
  #1\label{#2}
}
\makeatother
\hypersetup{urlcolor=blue, citecolor=red}

\title[Approximate controllability for Navier--Stokes equations] 
      {Approximate controllability for Navier--Stokes equations in {\rm 3D} Cylinders under Lions boundary conditions by an explicit saturating set}

\author[Duy Phan]{}


 \email{duy.phan-duc@uibk.ac.at}

\thanks{The author is supported by Universit\"at Innsbruck. The author would like to appreciate S\'ergio S. Rodrigues (RICAM Linz Austria) for his fruitful discussions to improve this work and Sy Nguyen-Ky (HAMK Finland) for all helpful figures.}

\thanks{$^*$ Corresponding author: duy.phan-duc@uibk.ac.at}

\begin{document}
\maketitle

\centerline{\scshape Duy Phan$^*$}
\medskip
{\footnotesize
 \centerline{Institut f\"ur Mathematik, Universit\"at Innsbruck}
   \centerline{Technikerstra\ss e 13/7, A-6020 Innsbruck, Austria.}
} 

\medskip

\bigskip


\begin{abstract}
An explicit saturating set consisting of eigenfunctions of Stokes operator in general 3D Cylinders is proposed. The existence of saturating sets implies the approximate controllability for Navier--Stokes equations in $\rm 3D$ Cylinders under Lions boundary conditions. 
\end{abstract}

\section{Introduction}
We consider the incompressible 3D Navier--Stokes equation in\\$(0,\,T)\times\Omega$, under Lions boundary conditions,%
\begin{subequations}\label{sys-u-flat}
 \begin{align}
 \p_t u+\langle u\cdot\nabla\rangle u-\nu\Delta u+\nabla p+h&=0, &\diver u &=0,\label{sys-u-flat-eq}\\
 \left.\begin{pmatrix} u\cdot\nnn\\ \curl u-(\nnn \cdot \curl u)\nnn\end{pmatrix}\right|_{\p\Omega}&=\begin{pmatrix}0\\0\end{pmatrix}, &u(0,\,x)&=u_0(x),
\end{align}
\end{subequations}
where~$\Omega\subset\R^3$ is an arbitrary three-dimensional cylinder
\[
 \Omega = (0,L_1) \times (0,L_2) \times \left(  \frac{2L_3}{2 \pi} \mathbb{S}^1  \right),
\]
whose boundary is denoted by $\p{\Omega} \coloneqq \Big(\big(\{0,L_1\} \times (0,L_2)\big))\cup\big((0,L_1) \times \{0,L_2\}\big)\Big) \times
   \left(\frac{{2L_3}}{2 \pi} \mathbb{S}^1 \right)$. As usual $u=(u_1,u_2,u_3)$ and~$p$,
defined for $( t,x_1,x_2,x_3)\in I\times\Omega$, are respectively the unknown velocity field and
pressure of the fluid, $\nu>0$ is the viscosity, the operators~$\nabla$ and~$\Delta$ are respectively the
well known gradient and Laplacian in the space variables $(x_1,x_2,x_3)$,
$\langle u\cdot\nabla\rangle v$ stands for $(u\cdot\nabla v_1,u\cdot\nabla v_2,u\cdot\nabla v_3)$,
$\diver u\coloneqq \sum_{i=1}^3 \p_{x_i}u_i$,
$\curl u \coloneqq \left(\p_{x_2} u_3 - \p_{x_3}u_2,~\p_{x_3} u_1 -\p_{x_1} u_3, \p_{x_1} u_2 - \p_{x_2}u_1 \right)^\top $,
the vector~$\nnn$ stands for the outward unit normal vector to~$\p\Omega$, and $h$ is a fixed function. 

Notice that this is equivalent to take appropriate mixed Lions--periodic boundary conditions in the
infinite channel~$\mathrm R_{\mathrm C} =  (0,L_1) \times (0,L_2) \times \R$:
\begin{subequations} \label{eq-BoundCond}
\begin{align}
 \begin{pmatrix}u\cdot\nnn\\
   \curl u-(\nnn\cdot\curl u)\nnn
  \end{pmatrix}&=0,
 \quad\mbox{on}\quad\Bigl(\bigl(\{0,L_1\} \times (0,L_2)\bigr))\cup\bigl((0,L_1) \times \{0,L_2\}\bigr)\Bigr) \times\R,\\
   u(x_1,x_2,x_3)&=u(x_1,x_2,x_3+2L_3),\quad (x_1,x_2,x_3)\in\mathrm R_{\mathrm C}.
\end{align}
\end{subequations}

The problem can be described as the model where the fluid is contained in a long (infinite) 3D channel with Lions boundary conditions on the bottom and the top of the channel, and with the periodicity assumption on the long (infinite) direction. Lions boundary condition is a particular case of Navier boundary conditions. For works and motivations concerning Lions and Navier boundary conditions (in both 2D and 3D cases) we
refer to~\cite{XiaoXin07,XiaoXin13,PhanRod17,PhanRod18JDCS, ChemetovCiprianoGavrilyuk10,Kelliher06,IlyinTiti06} and references therein.

\medskip
We set the spaces 
\begin{align*}
H &\coloneqq\{u\in L^2(\Omega,\,\R^3)\mid\diver u=0\mbox{ and }(u\cdot\nnn)\rest{\p\Omega}=0\}, \\
 V&\coloneqq\{u\in H^1(\Omega,\,\R^3)\mid u\in H\}.
\end{align*}
We consider $H$, endowed with the norm inherited from $L^2(\Omega,\,\R^3)$, as a pivot space,
that is, $H=H'$. 
For $u,v,w\in V$, we define
\begin{align}
 A&\colon V\to V',&\langle A u,v\rangle_{V',V}&\coloneqq \nu(\curl u,\curl v)_{L^2(\Omega,\,\R^3)},\label{LStokes}\\
 B&\colon V\times V\to V',&\langle B(u,v),w\rangle_{V',V}&\coloneqq -\int_\Omega(\langle u\cdot\nabla\rangle w)\cdot v\,\ed\Omega.\label{Bop}
\end{align}
The domain of operator $A$ is denoted as 
\begin{align*}
\D(A)\coloneqq\{u\in H^2(\Omega,\,\R^3)\mid u\in H, \quad\curl u-(\nnn\cdot\curl u)\nnn\rest{\p\Omega}=0\}.
\end{align*}
We will refer to~$A$ as the Stokes operator, under Lions boundary conditions. Further,
we have the continuous, dense, and compact inclusions~$\D(A)\xhookrightarrow{\rm d,c}V\xhookrightarrow{\rm d,c}H$. 
\begin{remark}
The notation $S\xhookrightarrow{} R$ above means that the inclusion $S \subseteq R$ is continuous. The letter ``$\rm d$'' (respectively ``$\rm c$'') means that, in addition, the inclusion is also dense (respectively compact).  
\end{remark}
\begin{remark}
Under the definition of $A$ in \eqref{LStokes}, it turns out that~$\D(A)=\{u\in H\mid Au\in H\}$ is the domain of~$A$. Indeed, by using the formula 
\begin{align*}
\nabla(\diver u) - \curl(\curl u) = \Delta u,  
\end{align*}
we can verify that if $u \in \D(A)$, $\diver (\Delta u) = 0$ and $(\Delta u \cdot  \nnn) = 0$ on $\p \Omega$. 
\end{remark}

The eigenvalues of $A$,
repeated accordingly with their multiplicity, form an increasing sequence $(\underline\lambda_k)_{k\in\N_0}$,
\[
 0<\underline\lambda_1\le\underline\lambda_2\le\underline\lambda_3\le\underline\lambda_4\le\dots,
\]
with $\underline\lambda_k$ going to $+\infty$ with $k$. 

\medskip
We can rewrite system~\eqref{sys-u-flat} as an evolutionary system
\begin{equation*}
 \dot u+A u+B(u,u)+h=\eta,\quad u(0)=u_0,
\end{equation*}
in the subspace~$H$ of divergence free vector fields
which are tangent to the boundary. We may suppose that~$h$ and~$\eta$ take their values in~$H$ (otherwise we just take their orthogonal projections onto~$H$). Denoting by~$\Pi$ the orthogonal projection in $L^2(\Omega,\,\R^3)$ onto $H$, for~$u,v\in\D(A)$ 
we may write~$Au\coloneqq\Pi(\nu\Delta u)$,
and $B(u,\,v)\coloneqq\Pi(\langle u\cdot\nabla\rangle v)$.

\subsection{Saturating sets and approximate controllability}\label{sS:saturH}

In the pioneering work~\cite{AgraSary05}, the authors introduced a method which led to the controllability of finite-dimensional Galerkin approximations of
the~$\mathrm{2D}$ and~$\mathrm{3D}$ Navier--Stokes system, and to the approximate controllability of the~$\mathrm{2D}$ Navier--Stokes system, by means of \\
low modes/degenerate forcing.

Hereafter~$U\subseteq H$ will stand for a linear subspace of~$H$, and we denote
\[
 \BB(a,b)\coloneqq B(a,b)+B(b,a), \qquad \text{for} \quad (a,~b) \in U \times U. 
\]

\begin{definition}\label{D:FF-l}
Let~$\CC=\{W_k\mid k\in\{1,\,2,\,\dots,\,M\}\}$ and let~$E$ be a finite-dimensional space so that~$\CC\subset E\subset U$. The finite-dimensional subspace $\FF_{\tt L}(E)\subset U$ is given by
\[
\FF_{\tt L}(E)\coloneqq E+\linspan\{\BB(a,b)\mid a\in\CC,\,b\in E,\,\mbox{ and } (B(a,a),B(b,b))\in H\times H\}\textstyle\bigcap  U.
\]
\end{definition}

\begin{definition}\label{D:satur-l}
A given finite subset~$\CC=\{W_k\mid k\in\{1,\,2,\,\dots,\,M\}\}
 \subset U$
 is said $({\tt L},U)$-saturating if for the following sequence of subspaces
 ~$\GG^j\subset U$, defined recursively by
 \begin{align*}
  \GG^0&\coloneqq\linspan\CC,\qquad \GG^{j+1}\coloneqq\FF_{\tt L}(\GG^{j}),
\end{align*}
we have that the union~$\mathop{\bigcup}\limits_{j\in\N}\GG^j$ is dense in~$H$.
\end{definition}

In~\cite[Section~4]{AgraSary06} the authors presented an explicit saturating set for the~$\mathrm{2D}$ Navier--Stokes system.
We would like to refer also to
the works~\cite{Romito04,HairerMatt06,EMat01}, where the notion of saturating set was used to
derive ergodicity for the Navier--Stokes system under degenerate stochastic forcing (compare the sequence of subsets $\ZZ_n$ in~\cite[section~4]{HairerMatt06} with the
sequence of subsets $\KK^n$ in~\cite[section~8]{AgraSary05}).

In the pioneering work~\cite{AgraSary05} the set~$U$ in~\eqref{D:satur-l} is taken to be~$\D(A)$, the same is done
in~\cite{AgraSary06,Rod-Sev05,Rod06,Shirikyan06}. Later, in~\cite{Rod-Thesis08,Rod-wmctf07,PhanRod-ecc15}, $U$ is taken as~$V$ in order to deal either with Navier-type boundary
conditions or with internal controls supported in a small subset.

In~\cite{Shirikyan06}, the method introduced in~\cite{AgraSary05} was developed in the case where the well-posedness of the Cauchy problem is not known. Though the author focuses on no-slip boundary conditions, i.e. $u\rest{\p\Omega}=0$, the results also hold for other boundary conditions. The author considered the case of periodic boundary conditions, and presented an explicit saturating set~$\CC$ (for the case of~$(1,1,1)$-periodic vectors). This saturating set consisted of ~$64$ eigenfunctions of the Stokes operator (i.e., the Laplacian). For a general period~$q=(q_1,q_2,q_3)\in\R_0^3$ the existence of a saturating set was also proven in~\cite[Section 2.3, Theorem~2.5]{Shirikyan06} even though the form of the saturating set was less explicit.

In \cite{PhanRod18JDCS}, the approximate controllability also follows from the existence of a~$({\tt L},\D(A))$-saturating set. For any given length triplet~$L=(L_1,L_2,L_3)$ of a {\rm 3D} rectangle, we presented an explicit $({\tt L},\D(A))$-saturating set~$\CC$ for the~$\mathrm{3D}$ rectangle~$\Omega=(0,L_1)\times(0,L_2)\times(0,L_3)$ (which will be recalled below). The elements of~$\CC_{\rm R}$ are~$81$ eigenfunctions of the Stokes operator under Lions boundary conditions. 

In various works of this topic, to tackle different types of boundary conditions as well as domains, some different definitions of saturating set has been proposed. Here we follow the definition of saturating set as in the previous work (see \cite{PhanRod18JDCS}) because it leads to some advantages in computations.

For further results concerning the controllability and approximate controllability of Navier--Stokes (and also other) systems by a control with low finite-dimensional range (independent of the viscosity coefficient) in several domains (including the $\mathrm{2D}$ Sphere and Hemisphere) we refer the readers to~\cite{Nersisyan10,Nersisyan11,Nersesyan10,Shirikyan08,Shirikyan_Evian07,Shirikyan07,Sarychev12,AgraSary06,AgraSary08,AgrKukSarShir07}.
We also mention Problem~VII raised by A.~Agrachev in~\cite{Agrachev13_arx} where the author inquired about the achievable controllability properties for 
controls taking values in a saturating set whose elements are localized/supported in a small subset~$\omega\subset\Omega$. The existence of such saturating sets is an open
question (except for~$\mathrm{1D}$ Burgers in ~\cite{PhanRod-ecc15}). The controllability properties implied by such saturating set is an open question. There are some negative
results, as for example in the case we consider the~$\mathrm{1D}$ Burgers
equations in~$\Omega=(0,1)$ and take controls in~$L^2(\omega,\R)$, $w\subset\Omega$, the approximate controllability fails to hold. Instead, to drive the system
from one state~$u_0=u(0)$ at time~$t=0$ to another one~$u_T=u(T)$ at time~$t=T$, we may need $T$ to be big enough. Though we do not consider localized controls here, we refer
the reader to the related results in~\cite{FerGue07,Shirikyan17} and references therein. 

\subsection{The main contribution}
We will present an explicit saturating set in the case of three-dimensional cylinder domain. The saturating set  consists of finite number of eigenfunctions of Stokes operator (see Theorem \ref{T:satur3Dcyl} hereafter). The saturating set has 355 elements (or a simpler version with 260 elements in corollary \ref{cor-sat-cyl}). In some particular cases, it may exist other saturating sets with less elements. However, we want to emphasize that our goal is not to find a saturating set with minimal number of elements. In all cases, the existence of a $({\tt L},\D(A))$-saturating set must be independent of the viscosity coefficient $\nu$. In particular, the linear space $\GG^1$, where the control $\eta$ takes its value, does not change with $\nu$. 

To construct a saturating set, we firstly introduce a system of eigenfunctions in Section \ref{subsec:eigfun}. In this type of domain, we have two types of eigenfunctions $ Y^{j(k),k}$ and $ Z^{j(k),k}$. The form of eigenfunctions $ Y^{j(k),k}$ are analogous to the ones in 3D Rectangles. Nevertheless the apperance of another type of eigenfunctions $Z^{j(k),k}$ yields to some difficulties. The construction of all eigenfunctions $ Z^{j(k),k}$ is based on the expression of $\left( Z^{j(k),k} \cdot \nabla   \right) Y^{j(m),m} + \left( Y^{j(m),m} \cdot \nabla   \right) Z^{j(k),k}$ (see Section \ref{subsec_YkZm}). To construct the eigenfunctions $ Z^{j(k),k}$, Lemma \ref{lemma-lin-indp_Z} (a similar version of Lemma 3.1 in \cite{PhanRod18JDCS}) is not enough to prove the linear independence. Therefore another Lemma \ref{lemma-lin-indp2} will be introduced and used mostly in the proof. The Lemma \ref{lemma-lin-indp2} is a fruitful tool to prove the linear independence property in most of the cases (for example in Step \ref{Step352}, equation \eqref{eq-lem-fail}) where Lemma \ref{lemma-lin-indp_Z} fails to prove the property. Another remark comes from the difference of boundary condition in the third direction from other directions. 
We notice that the procedure can be applied analogously in the first two directions because we consider Lions boundary conditions in the first two directions (see the proof in Section \ref{subsec-proof.genR12}). However, the third direction must be addressed separately in Section \ref{subsec-proof.genR3} because we consider the periodicity assumption in the third direction. In conclusion we believe that the proof in the case of ${\rm 3D}$ Cylinder is inspired from the case of ${\rm 3D}$ Rectangle but it cannot follow line by line.

\bigskip
The rest of the paper is organized as follows. In Section \ref{sec-pre}, we recall some results of the approximate controllability for 3D Navier-Stokes equations under Lions boundary conditions. An explicit saturating set in the case of three-dimensional Rectangle will be revisited in Section~\ref{sec: Recall3DRect}. In Section~\ref{sec: Sat3DCyl}, we construct a $({\tt L},\D(A))$-saturating set in the case of three-dimensional cylinders. The proof of the main Theorem \ref{T:satur3Dcyl} needs more complicated computations. The core ideas of the proof will be presented immediately. More detailed computations will be found in Section \ref{sec: ProofSat3DCyl}.

\section{Preliminaries} \label{sec-pre}

\subsection{Approximate controllability}
Hereafter $u_0\in V$, ~$h\in L^2_{\rm loc}(\R_0,H)$, and~$E\subset\D(A)$ is a finite-dimensional subspace.
Let us consider the system
\begin{equation}\label{sys_u_E}
 \dot u+A u+B(u,u)+h=\eta,\quad u(0)=u_0, 
\end{equation}
where the control~$\eta$ takes its values in~$E$.

For simplicity we will denote
\[
 I_T\coloneqq(0,T),\quad\mbox{and}\quad \overline{I_T}\coloneqq [0,T],\qquad T>0.
\]

\begin{definition}
 Let $T$ be a positive constant. System~\eqref{sys_u_E} is said to be $E$-approximately controllable in time~$T$
 if for any~$\e > 0$ and any pair~$(u_0,\hat u)\in V\times \D(A)$, there exists a control function
$\eta\in L^\infty(I_T, E)$ and a corresponding solution\\ $u\in C(\overline{I_T},V)\bigcap L^2(I_T, \D(A))$, such that
$\norm{u(T)-\hat u}{V} <\e$.
\end{definition}

We recall the Main Theorem in \cite{PhanRod18JDCS} which shows the approximate controllability of 3D Navier-Stokes system from the existence of a $({\tt L},\D(A))$-saturating set. 
\begin{theorem}
Let $(u_0,\hat u)\in V\times V$, ~$\varepsilon>0$, and~$T > 0$. If~$\CC$ is a $({\tt L},\D(A))$-saturating set,
then we can find a control~$\eta\in L^\infty((0,T),\GG^1)$ so that the solution of system~\eqref{sys_u_E} satisfies
$\norm{u(T)-\hat u}{V}<\varepsilon$.
\end{theorem}
\begin{remark}
In \cite{Shirikyan06}, the author introduced a definition of $(\rm B, D(A))$ saturating set and proved that the existence of a $(\rm B, D(A))$ saturating set implies the approximate controllability of the 3D Navier-Stokes systems, at time $T>0$. In \cite{PhanRod18JDCS} and this work, we are using another definition of saturating set, so-called $(\rm L, D(A))$ saturating set. The main advantage to use this definition is in the computation below. 
\end{remark}

\subsection{An explicit saturating set in 3D Rectangles}
\label{sec: Recall3DRect} In this section, we recall a $({\tt L},\D(A))$-saturating set containing a finite number of suitable eigenfunctions of the Stokes operator~$A$ in the $\mathrm{3D}$ rectangle
\[
R=(0,L_1) \times (0,L_2) \times(0,L_3)
\]
under Lions boundary conditions (see~\eqref{LStokes}), where $L_1$, $L_2$, and~$L_3$ are positive numbers. 

For a given $k \in \N^3$, let $\#_0(k)$ stand for the number of vanishing components of $k$. The index $j(k) \in \N$ is such that $1 \le j(k) \le 2 - \#_0(k)$. For example, if $k = (2,1,0)$, $\#_0(k) = 1$ and $j(k) \in \{ 1 \}$; or if  $k = (2,1,2)$, $\#_0(k) = 0$ and $j(k) \in \{ 1,2 \}$. 
\begin{definition}
For $x, y \in \R^3\setminus\{(0,0,0)\}$, the scalar product $(\cdot, \cdot)_{[L]}$ is defined as
\begin{align}
(x,\,y)_{[L]} &\coloneqq \frac{ x_1 y_1}{L_1} + \frac{ x_2 y_2}{L_2} + \frac{x_3 y_3}{L_3}.
\end{align}
For $k \in \N^3$, we denote the orthogonal space $\{k\}^{\perp_{[L]}}_0$ as follows
\begin{align}
\{k\}^{\perp_{[L]}}_0&\coloneqq\{z\in\R^3\setminus\{(0,0,0)\}\mid \left( z,\, k \right)_{[L]} =0,
\mbox{ and }z_i=0\mbox{ if }k_i=0\}.
\end{align}

\end{definition}

A complete system of eigenfunctions $\left\{  Y_{k}  \right\} $ is given by
\begin{subequations}\label{FamilyEig3DRect}
\begin{equation} \label{FamilyEig3DRect.Y}
 Y^{j\left(k\right),k} \coloneqq \begin{pmatrix}
w_1^{{j\left(k\right),k}} \sin \left( \frac{k_1 \pi x_1}{L_1} \right) \cos \left( \frac{k_2 \pi x_2}{L_2} \right) \cos \left( \frac{ k_3 \pi x_3}{L_3} \right)\\
w_2^{{j\left(k\right),k}} \cos \left( \frac{k_1 \pi x_1}{L_1} \right) \sin \left( \frac{k_2 \pi x_2}{L_2} \right) \cos \left( \frac{k_3 \pi x_3}{L_3} \right)\\ 
w_3^{{j\left(k\right),k}} \cos \left( \frac{k_1 \pi x_1}{L_1} \right) \cos \left( \frac{k_2 \pi x_2}{L_2} \right) \sin \left( \frac{k_3 \pi x_3}{L_3} \right)
 \end{pmatrix},\quad \#_0(k)\le1,
 \end{equation}
 with
 \begin{equation} 
\{ w ^{j(k),k}\mid j(k)\in\{1,2-\#_0(k)\}\}\subset\{k\}^{\perp_{[L]}}_0
\end{equation}
a linearly independent and orthogonal family.
\end{subequations}

Notice that $2 - \#_0(k)$ is the dimension of the subspace $\{ k \}^{\perp_{[L]}}_0$ and that the orthogonality of the
family~$\{ w ^{j(k),k}\mid j(k)\in\{1,2-\#_0(k)\}\}$ implies that 
the family in~\eqref{FamilyEig3DRect.Y} is also orthogonal. 

We recall the result in \cite[Theorem 3.1]{PhanRod18JDCS} about saturating set in 3D rectangles 

\begin{theorem}\label{T:satur3Drect}
The set~$~\CC_{\rm Y}  \coloneqq\left\{ Y^{j(n),n} \left|\begin{array}{l}n\in\N^3,\quad 0 \le n_i \le 3,\\ \#_0(n) \le 1,\quad j(n) \in \{ 1, 2-\#_0(n) \}\end{array} \right. \right\}$
is $({\tt L},\D(A))$-saturating. 
\end{theorem} 
The theorem plays a remarkable role to prove the saturating set in 3D-cylinder case in Section \ref{sec: Sat3DCyl}. Particularly, the inclusion \eqref{Sat_Inc_CY} will be obtained directly from this theorem. 

\section{A saturating set in the 3D-cylinder case}
\label{sec: Sat3DCyl}
\subsection{A system of eigenfunctions} \label{subsec:eigfun}
The system of eigenfunctions in the 3D-cylinder case consists two types of eigenfunctions $Y^{j\left(k\right),k}$ and $ Z^{j\left(k\right),k}$. The first type of eigenfunctions was defined before in \eqref{FamilyEig3DRect.Y}. Another type of eigenfunctions $Z^{j\left(k\right),k}$ is defined as follows 
\begin{align} \label{FamilyEig3DCylT2}
  Z^{j\left(k\right),k} &= \begin{pmatrix}
w_1^{{j\left(k\right),k}} \sin \left( \frac{k_1 \pi x_1}{L_1} \right) \cos \left( \frac{k_2 \pi x_2}{L_2} \right) \sin \left( \frac{k_3 \pi x_3}{L_3} \right)\\
w_2^{{j\left(k\right),k}} \cos \left( \frac{k_1 \pi x_1}{L_1} \right) \sin \left( \frac{k_2 \pi x_2}{L_2} \right) \sin \left( \frac{k_3 \pi x_3}{L_3} \right)\\ 
-w_3^{{j\left(k\right),k}} \cos \left( \frac{k_1 \pi x_1}{L_1} \right) \cos \left( \frac{k_2 \pi x_2}{L_2} \right) \cos \left( \frac{k_3 \pi x_3}{L_3} \right)
 \end{pmatrix},
\end{align}
either with $\#_0(k)\le1$, or with $\#_0(k)=2$ and $k_3=0$, or with $\#_0(k)=3$. 

Furthermore we assume that the vectors $w ^{j(k),k}$ are chosen satisfying
 \begin{itemize}
 \item if~$\#_0(k)= 3$, $w ^{j(k),k} = (0,0,-1)$. Then $Z^{(0,0,0)} = (0,0,1)^\top$.  
 \item if~$\#_0(k)= 2$ and~$k_3=0$, $w ^{j(k),k}=w ^{1,k}=(0,0,w ^{1,k}_3)$, with~$w ^{1,k}_3\ne0$,
  \item if~$\#_0(k)\le 1$, then
$ w ^{j(k),k} \in \{k\}^{\perp_{[L]}}_0. 
$ 
\end{itemize}
\begin{remark}
Even though the domain~$\Omega$ is different, namely~$x\in R=(0,L_1) \times (0,L_2) \times(0,L_3)$ and
$x\in\Omega \sim(0,L_1) \times (0,L_2) \times(0,2L_3)$, the eigenfunctions $Y^{j\left(k\right),k}$ are the same. 
\end{remark}

These functions of the forms \eqref{FamilyEig3DRect.Y} and \eqref{FamilyEig3DCylT2} are eigenfunctions of the shifted Stokes operator~$A$ in~$\Omega$ under the boundary conditions \eqref{eq-BoundCond}. Indeed, it is clear that they are eigenfunctions of the usual Laplacian operator. So it remains to check that they are divergence-free and satisfy the boundary conditions \eqref{eq-BoundCond}.

The divergence free condition follows from the choices of the vectors~$w^{j(k),k}$. It is also clear that $u\cdot\nnn$ vanishes at the boundary~$\p\Omega$. Finally we can see that the $\curl$ is normal to the boundary, from the expressions
\begin{align*}
\curl Y ^{j(k),k}
&= -\pi\begin{pmatrix}\left(\textstyle\frac{k_2}{L_2} w ^{j(k),k}_3-\textstyle\frac{k_3}{L_3} w ^{j(k),k}_2\right)
\cos(\textstyle\frac{k_1\pi x_1}{L_1})\sin(\textstyle\frac{k_2\pi x_2}{L_2})\sin(\textstyle\frac{k_3\pi x_3}{L_3})\\
\left(\textstyle\frac{k_3}{L_3} w ^{j(k),k}_1-\textstyle\frac{k_1}{L_1} w ^{j(k),k}_3\right)
\sin(\textstyle\frac{k_1\pi x_1}{L_1})\cos(\textstyle\frac{k_2\pi x_2}{L_2})\sin(\textstyle\frac{k_3\pi x_3}{L_3})\\
\left(\textstyle\frac{k_1}{L_1} w ^{j(k),k}_2-\textstyle\frac{k_2}{L_2} w ^{j(k),k}_1\right)
\sin(\textstyle\frac{k_1\pi x_1}{L_1})\sin(\textstyle\frac{k_2\pi x_2}{L_2})\cos(\textstyle\frac{k_3\pi x_3}{L_3})
\end{pmatrix},\\
\curl Z^{j(k),k} &= \pi \begin{pmatrix}
\left( \frac{k_2}{L_2}  w_3^{{j\left(k\right),k}} -  \frac{k_3}{L_3}  w_2^{{j\left(k\right),k}} \right)  \cos \left( \frac{k_1 \pi x_1}{L_1} \right) \sin \left( \frac{k_2 \pi x_2}{L_2} \right) \cos \left( \frac{k_3 \pi x_3}{L_3}    \right) \\
\left( \frac{k_3}{L_3}  w_1^{{j\left(k\right),k}} -  \frac{k_1}{L_1}  w_3^{{j\left(k\right),k}} \right)  \sin \left( \frac{k_1 \pi x_1}{L_1} \right) \cos \left( \frac{k_2 \pi x_2}{L_2} \right) \cos \left( \frac{k_3 \pi x_3}{L_3}    \right)  \\
\left( \frac{k_2}{L_2}  w_1^{{j\left(k\right),k}} -  \frac{k_1}{L_1}  w_2^{{j\left(k\right),k}} \right)  \sin \left( \frac{k_1 \pi x_1}{L_1} \right) \sin \left( \frac{k_2 \pi x_2}{L_2} \right) \sin \left( \frac{k_3 \pi x_3}{L_3}    \right) 
\end{pmatrix},
\end{align*}
which we can derive by direct computations. For example, at the lateral boundary~$x_1=L_1$, that is,
for $x\in \{L_1\} \times (0,L_2)\times\frac{2L_3}{2 \pi} \mathbb{S}^1$, we have $\nnn=(1,0,0)$ and
\begin{align*}
\curl Y ^{j(k),k}
&= -\pi\begin{pmatrix}\left(\textstyle\frac{k_2}{L_2} w ^{j(k),k}_3-\textstyle\frac{k_3}{L_3} w ^{j(k),k}_2\right)
\sin(\textstyle\frac{k_2\pi x_2}{L_2})\sin(\textstyle\frac{k_3\pi x_3}{L_3})\\
0\\
0
\end{pmatrix},\\
\curl Z^{j(k),k} &= \pi \begin{pmatrix}
\left( \frac{k_2}{L_2}  w_3^{{j\left(k\right),k}} -  \frac{k_3}{L_3}  w_2^{{j\left(k\right),k}} \right)  \sin \left( \frac{k_2 \pi x_2}{L_2} \right) \cos \left( \frac{k_3 \pi x_3}{L_3}    \right) \\
0  \\
0
\end{pmatrix},
\end{align*}
which show that~$\curl Y ^{j(k),k}$ and~$\curl Z^{j(k),k}$ have the same direction as the normal vector~$\nnn$.

\begin{lemma}
The system of eigenfunctions
\begin{align*}
&\{ Y^{j(k), k},~ Z^{j(k),k} \mid k\in \N^3\mbox{ and }\#_0(k) \le 1 \}\\
& {\textstyle\bigcup} \{Z^{1,k} \mid
k \in \N^3,~\#_0(k)=2,\mbox{ and }k_3=0\}  
~ {\textstyle\bigcup} \left\{Z^{(0,0,0)} =(0,0,1)^\top \right \}
\end{align*}
is complete. 
\end{lemma}
\begin{proof}
Recalling that, for~$r>0$,
\[
 \{\sin(\textstyle\frac{k\pi x_1}{r})\mid k\in\N_0\}\quad\mbox{and}\quad\{\cos(\textstyle\frac{k\pi x_1}{r})\mid k\in\N\}
\]
are two complete systems in~$L^2((0,r),\R)$. And
\[
 \{\sin(\textstyle\frac{k\pi x_1}{r})\mid k\in\N_0\}\bigcup\{\cos(\textstyle\frac{k\pi x_1}{r})\mid k\in\N\}
\]
is a complete system in~$L^2((0,2r),\R)$. Then the proof can be done by following the arguments in~\cite[Section 6.6]{PhanRod17}. We skip the details.
\end{proof}

Now we can present the saturating set.
\begin{theorem}\label{T:satur3Dcyl}
The set of eigenfuntions~
\begin{align*}
\CC \coloneqq &\left\{ Y^{j(n),n} \mid~n\in\N^3,~\#_0(n)\le1,~n_i\le 4,~j(n) \in \{ 1, 2-\#_0(n) \}\right\} \\
& {\textstyle\bigcup}  \left\{ Z^{j(n),n} \mid~n\in\N^3,~\#_0(n)\le1,~n_i\le {4},~j(n) \in \{ 1, 2-\#_0(n) \}\right\} \\
& {\textstyle\bigcup} \left\{ Z^{j(n),n} \mid~n \in \N^3,~\#_0(n)=2,~n_3=0 \right\} ~ {\textstyle\bigcup}  \left\{Z^{(0,0,0)} =(0,0,1)^\top \right \}
\end{align*}
is $({\tt L},\D(A))$-saturating.  
\end{theorem} 
\begin{proof}
Firstly, let us denote $\N_4 \coloneqq \{q \in \N \mid q \ge 4 \}$.
We recall the index subsets $\sS^q_{\mathrm R},\,\RR^q_m,\,\LL^q_{m_1,m_2}$ defined in the proof of rectangle case (see \cite[Section 3.4]{PhanRod18JDCS}) for $q \in \N_4 \coloneqq \{q \in \N \mid q \ge 4 \}$.
\begin{equation}\label{setsSC_R}
\begin{split}
\sS^q_{\rm R} &\coloneqq \left\{ n\in\N^3 \mid 0 \le n_i \le q,~ \#_0(n) \le 1 \right\},\\
\CC^q_{\rm R} &\coloneqq \left\{ Y^{j(n),n} \mid~n \in \sS^q_{\rm R},~ j(n) \in \{ 1, 2-\#_0(n) \}  \right\},\\
\RR^q_m &\coloneqq \left\{ n \in \sS^q_{\rm R}  \mid n_m = q,~0 \le n_i \le q-1~\mbox{for }i \neq m \right\},\\
\LL^q_{m_1,m_2} & \coloneqq \left\{ n \in \sS^q_{\rm R}   \mid ~n_{m_1} = q= n_{m_2},~m_1\ne m_2,~0 \le n_i \le q-1,~i \notin \{m_1,m_2\} \right\}.
\end{split}
\end{equation}

Next, we define some new sets
\begin{equation}\label{setsSC_C}
\begin{array}{rcl}
\sS^q_{\rm C} &\coloneqq& \sS^q_{\rm R} \cup \{ (n_1,0,0), (0,n_2,0) \mid 0 < n_1 \le q,~0 < n_2 \le q \},\\
\CC^q_{\rm C} &\coloneqq& \left\{ Y^{j(n),n},\, Z^{j(n),n} \mid~n \in \sS^q_{\rm R},~ j(n) \in \{ 1, 2-\#_0(n) \}  \right\}\\
 & &\bigcup \left\{ Z^{j(n),n} \mid~n \in \sS^q_{\rm C} \setminus  \sS^q_{\rm R},~ j(n) =1 \right\} \bigcup \left\{ Z^{(0,0,0)}\right\}.
 \end{array}\hspace*{-1em}
\end{equation}

We  can see that Theorem~\ref{T:satur3Dcyl} is a corollary of the following inclusions
\begin{equation}\label{Sat_Inc_C}
 \CC^q_{\rm C}\subseteq \GG^{q-1},\qquad \mbox{for all}\quad q\in\N_4.
\end{equation}
Let us decompose~$\CC^q_{\rm C}=\CC^q_{Y}\cup\CC^q_{Z}$ with
\begin{subequations}\label{splitCYZ}
 \begin{align}
\CC^q_{Y}&\coloneqq\left\{ Y^{j(n),n} \mid~n \in \sS^q_{\rm R},~ j(n) \in \{ 1, 2-\#_0(n) \}  \right\}\\
 \CC^q_{Z}&\coloneqq \left\{ Z^{j(n),n} \mid~n \in \sS^q_{\rm R},~ j(n) \in \{ 1, 2-\#_0(n) \}  \right\} \nonumber \\
 &\qquad {\textstyle\bigcup} \left\{ Z^{1,n} \mid
 ~n \in \sS^q_{\rm C}\setminus  \sS^q_{\rm R} \right\}  {\textstyle\bigcup} \left\{ Z^{(0,0,0)}\right\}. 
 \end{align}
\end{subequations}
Inspiring from the proof of Theorem \ref{T:satur3Drect}, we get
\begin{equation}\label{Sat_Inc_CY}
 \CC^q_{Y}\subseteq \GG^{q-1},\qquad \mbox{for all}\quad q\in\N_4.
\end{equation}
So it remains to prove that
\begin{equation}\label{Sat_Inc_CZ}
 \CC^q_{Z}\subseteq \GG^{q-1},\qquad \mbox{for all}\quad q\in\N_4,
\end{equation}
The inclusion \eqref{Sat_Inc_CZ} is proved by induction argument.
\\ \medskip
\textbf{Base step}
By definition we have that $\CC=\CC^4_{\rm C}\supset\CC^4_{Z}$ and~$\linspan\CC=\GG^0$. Therefore
\begin{equation}
 \mbox{Inclusion~\eqref{Sat_Inc_CZ} holds for }q=4. \label{BaseIndStep_C}
\end{equation}
\textbf{Induction Step} The induction hypothesis is
\addtocounter{equation}{1}
\begin{equation}\label{IH.C}
 \mbox{ $\CC^4_{Z}\subseteq \GG^{0}$ and the inclusion $\CC^q_{Z}\subseteq \GG^{q-1}$ holds true for a given $q\in\N_4$.}
\end{equation}
We want to prove that $\CC^{q+1}_{Z}\subseteq \GG^{q}$.

Notice that
\begin{equation} \label{subset-sq-3Dcyl}
 \begin{split}
\sS^{q+1}_{\rm C}  =  &\sS^{q}_{\rm C} \bigcup \left( \RR^{q+1}_1\cup\RR^{q+1}_2\cup\RR^{q+1}_3 \right) \bigcup \{ (q+1, 0, 0),\, (0,q+1, 0) \}\\ 
&\bigcup \left( \LL^{q+1}_{1,2} \cup \LL^{q+1}_{2,3} \cup \LL^{q+1}_{3,1} \right) \bigcup \{ (q+1, q+1, q+1) \}.
\end{split}
\end{equation} 
The decomposition \eqref{subset-sq-3Dcyl} is sketched in Figure \ref{fig_induction}. 

Based on this decomposition \eqref{subset-sq-3Dcyl},we introduce five Lemmas below.
\begin{lemma} \label{lem-genR3}
$Z^{j(n),n} \in \GG^q$ for all $ n \in \RR^{q+1}_3$.
\end{lemma}

\begin{lemma} \label{lem-genR12}
$Z^{j(n),n} \in \GG^q$ for all $ n \in\left( \RR^{q+1}_1\cup\RR^{q+1}_2  \right)$.
\end{lemma}

\begin{lemma} \label{lem-gen00}
$Z^{j(n),n} \in \GG^q$ for all $ n \in \{ (q+1, 0, 0),\, (0,q+1, 0) \}$.
\end{lemma}

\begin{lemma} \label{lem-genL}
$Z^{j(n),n} \in \GG^q$ for all $ n \in \left( \LL^{q+1}_{1,2} \cup \LL^{q+1}_{2,3} \cup \LL^{q+1}_{3,1} \right)$.
\end{lemma}

\begin{lemma} \label{lem-genq+1}
$Z^{\{1,2\},(q+1,q+1,q+1)} \subset \GG^q$.
\end{lemma}

\begin{figure}[ht] \centering 
\includegraphics[width=\textwidth]{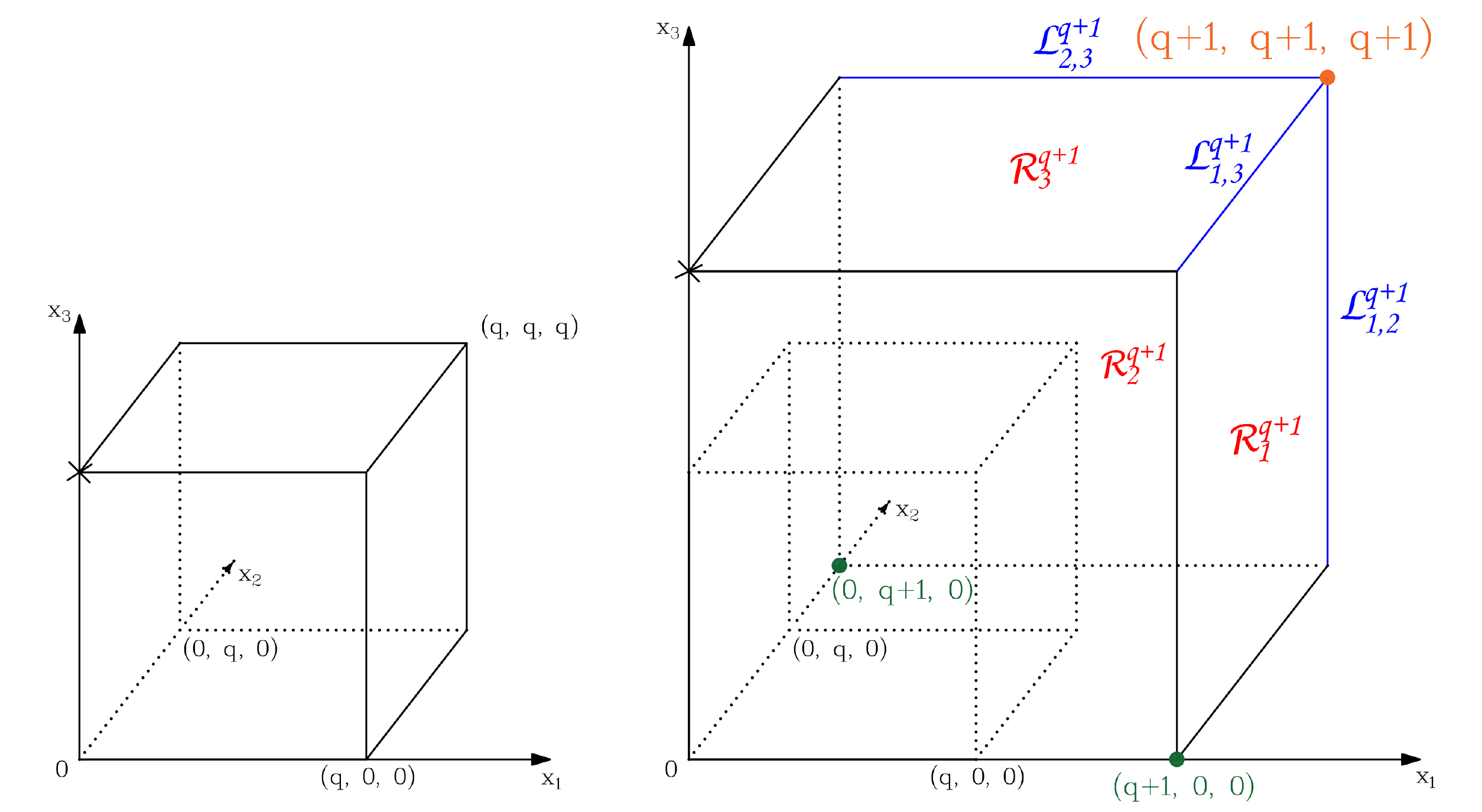}
\caption{Induction Step based on the decomposition \eqref{subset-sq-3Dcyl}.}
\label{fig_induction}
\end{figure}

Under the composition \eqref{subset-sq-3Dcyl}, we will sequentially construct eigenfunctions whose indices belong to each parts of the decomposition. 
The first part containing three rectangles $\RR^{q+1}_1,~\RR^{q+1}_2,~\RR^{q+1}_3$ will be considered in Lemmas \ref{lem-genR3} and \ref{lem-genR12}. We need to consider the third direction separately due to the difference of boundary condition in this direction. The Lemma \ref{lem-gen00} deals with two special eigenfunctions which did not appear in the proof of Theorem 3.1 in \cite{PhanRod18JDCS} (the indices are two green points in Figure \ref{fig_induction}). The Lemma \ref{lem-genL} tackles these eigenfunctions whose indices belong to three blue lines $\LL^{q+1}_{1,2},~\LL^{q+1}_{2,3},~\LL^{q+1}_{3,1}$ in Figure \ref{fig_induction}. The last Lemma \ref{lem-genq+1} will take on two eigenfunctions whose index is $(q+1, q+1, q+1)$ (the orange point in Figure \ref{fig_induction}). 
To prove the Lemmas \ref{lem-genR3}, \ref{lem-genR12}, and \ref{lem-genL}, we will process them by induction. 
The particular proofs of all Lemmas will be presented  in following Section \ref{sec: ProofSat3DCyl}. 

Following all results of Lemmas \ref{lem-genR3}--\ref{lem-genq+1}, we conclude that $\CC^{q+1}_{Z\rm C}\subseteq \GG^{q}$. By induction hypothesis and \eqref{BaseIndStep_C}, we can conclude that the inclusion \eqref{Sat_Inc_CZ} hold true, which implies the statement of Theorem \ref{T:satur3Dcyl}.
\end{proof}

\medskip
To prove Lemmas \ref{lem-genR3}--\ref{lem-genq+1}, we will next introduce some fruitful tools.

\subsection{The expression for $(Y^k \cdot \nabla) Z^m + (Z^m \cdot \nabla) Y^k $.} \label{subsec_YkZm}
Here we will present the expression for the coordinates of $\left( Y^{j(k),k} \cdot \nabla   \right) Z^{j(m),m} + \left( Z^{j(m),m} \cdot \nabla   \right) Y^{j(k),k}$
for given eigenfunctions in~\eqref{FamilyEig3DRect.Y} and \eqref{FamilyEig3DCylT2}. In order to shorten the following expressions and simplify the writing, we will write
\begin{align}Y^{k}=Y^{j\left(k\right),k},\quad Y^{m}=Y^{j\left(m\right),m},\quad  w ^{k}=w ^{j(k),k},\quad  \mbox{and}\quad w ^{m}=w ^{j(m),m}
\end{align}
by neglecting the indices $j(k),j(m)$. We will also denote
\[
 \textstyle {\rm C}_i(k_i)\coloneqq \cos \left( \frac{k_i \pi x_i}{L_i} \right)\quad\mbox{and}\quad {\rm S}_i(k_i)\coloneqq \sin \left( \frac{k_i \pi x_i}{L_i} \right),\qquad \text{for~~} i\in \{1,2,3\}.
\]

Proceeding as in the case of the rectangle (see \cite[Section 3.1]{PhanRod18JDCS}), we can obtain 
\begin{align*}
\left( Y^k \cdot \nabla   \right) Z^m &= \left( \begin{matrix}
Y^k \cdot w^m_1  \left( \ \begin{matrix}
\frac{m_1\pi}{L_1} {\rm C}_1(m_1) {\rm C}_2(m_2) {\rm S}_3(m_3) \\
 -\frac{m_2\pi}{L_2} {\rm S}_1(m_1) {\rm S}_2(m_2) {\rm S}_3(m_3) \\
 \frac{m_3\pi}{L_3} {\rm S}_1(m_1) {\rm C}_2(m_2) {\rm C}_3(m_3)
\end{matrix}   \right) \\
Y^k \cdot w^m_2  \left( \ \begin{matrix}
-\frac{m_1\pi}{L_1} {\rm S}_1(m_1) {\rm S}_2(m_2) {\rm S}_3(m_3) \\
\frac{m_2\pi}{L_2} {\rm C}_1(m_1) {\rm C}_2(m_2) {\rm S}_3(m_3) \\
 \frac{m_3\pi}{L_3} {\rm C}_1(m_1) {\rm S}_2(m_2) {\rm C}_3(m_3)
\end{matrix}   \right)\\
Y^k \cdot w^m_3  \left( \ \begin{matrix}
\frac{m_1\pi}{L_1}  {\rm S}_1(m_1) {\rm C}_2(m_2) {\rm C}_3(m_3) \\
\frac{m_2\pi}{L_2} {\rm C}_1(m_1) {\rm S}_2(m_2) {\rm C}_3(m_3) \\
 \frac{m_3\pi}{L_3}{\rm C}_1(m_1) {\rm C}_2(m_2) {\rm S}_3(m_3)
\end{matrix}   \right)
\end{matrix}
\right), \\
\left( Z^m \cdot \nabla   \right) Y^k &= \left( \begin{matrix}
Z^m \cdot w^k_1  \left( \ \begin{matrix}
\frac{k_1\pi}{L_1} {\rm C}_1(k_1) {\rm C}_2(k_2) {\rm C}_3(k_3) \\
 -\frac{k_2\pi}{L_2} {\rm S}_1(k_1) {\rm S}_2(k_2) {\rm C}_3(k_3) \\
-\frac{k_3\pi}{L_3} {\rm S}_1(k_1) {\rm C}_2(k_2) {\rm S}_3(k_3)
\end{matrix}   \right) \\
Z^m \cdot w^k_2  \left( \ \begin{matrix}
-\frac{k_1\pi}{L_1}  {\rm S}_1(k_1) {\rm S}_2(k_2) {\rm C}_3(k_3) \\
\frac{k_2\pi}{L_2} {\rm C}_1(k_1) {\rm C}_2(k_2) {\rm C}_3(k_3) \\
-\frac{k_3\pi}{L_3} {\rm C}_1(k_1) {\rm S}_2(k_2) {\rm S}_3(k_3)
\end{matrix}   \right)\\
Z^m \cdot w^k_3  \left( \ \begin{matrix}
-\frac{k_1\pi}{L_1}  {\rm S}_1(k_1) {\rm C}_2(k_2) {\rm S}_3(k_3) \\
-\frac{k_2\pi}{L_2} {\rm C}_1(k_1) {\rm S}_2(k_2) {\rm S}_3(k_3) \\
\frac{k_3\pi}{L_3}  {\rm C}_1(k_1) {\rm C}_2(k_2) {\rm C}_3(k_3)
\end{matrix}   \right)
\end{matrix}
\right),
\end{align*} 
We next denote coefficient $\beta^{\star_1\star_2\star_3}_{w^{k},m}$ as follows
\begin{align} \label{BetaStar}
\beta^{\star_1\star_2\star_3}_{w^{k},m} \coloneqq \frac{\pi}{8} \textstyle\left( \star_1~\frac{w_1^{k} m_1}{L_1}~\star_2~\frac{w_2^{k} m_2}{L_2}~\star_3~\frac{w_3^{k}m_3 }{L_3}\right),
\qquad\mbox{for}\quad {(\star_1, \star_2, \star_3) \in \{+,- \}^3.} 
\end{align}
For example, we have 
\begin{align*}
\beta^{+++}_{w^{k},m} \coloneqq \frac{\pi}{8} \textstyle\left( +\frac{w_1^{k} m_1}{L_1} + \frac{w_2^{k} m_2}{L_2} + \frac{w_3^{k}m_3 }{L_3}\right), \quad \beta^{-+-}_{w^{k},m} \coloneqq \frac{\pi}{8} \textstyle\left( -\frac{w_1^{k} m_1}{L_1} + \frac{w_2^{k} m_2}{L_2} - \frac{w_3^{k}m_3 }{L_3}\right). 
\end{align*}
By straightforward computation, we can find the expression for the coordinates of $\left( Y^k \cdot \nabla   \right) Z^m + \left( Z^m \cdot \nabla   \right) Y^k$ as follows  
\begin{subequations}\label{CoorYZ_Cyl}
\begin{gather}
\begin{split} \label{1stCoorMix}
&\left( \left( Y^k \cdot \nabla   \right) Z^m + \left( Z^m \cdot \nabla   \right) Y^k \right)_1\\
&\hspace{4em}= \left(w_1^{m}\beta^{+++}_{w^k,m} + w_1^{k} \beta^{+++}_{w^m,k} \right)
{\rm S}_1 (k_1 + m_1) {\rm C}_2(k_2 + m_2) {\rm S}_3 (k_3 + m_3)  \\
&\hspace{5em}+ \left( -w_1^{m}\beta^{+++}_{w^k,m} + w_1^{k} \beta^{+++}_{w^m,k}\right)
{\rm S}_1 (k_1 - m_1) {\rm C}_2(k_2 - m_2) {\rm S}_3 (k_3 - m_3)  \\
&\hspace{5em}+ \left( - w_1^{m}\beta^{++-}_{w^k,m} - w_1^{k} \beta^{++-}_{w^m,k} \right)
{\rm S}_1 (k_1 + m_1) {\rm C}_2(k_2 + m_2) {\rm S}_3 (k_3 - m_3)  \\
&\hspace{5em}+  \left( w_1^{m}\beta^{++-}_{w^k,m} - w_1^{k} \beta^{++-}_{w^m,k} \right)
{\rm S}_1 (k_1 - m_1) {\rm C}_2(k_2 - m_2) {\rm S}_3 (k_3 + m_3)  \\
&\hspace{5em}+ \left( w_1^{m}\beta^{+-+}_{w^k,m} + w_1^{k} \beta^{+-+}_{w^m,k} \right)
{\rm S}_1 (k_1 + m_1) {\rm C}_2(k_2 - m_2) {\rm S}_3 (k_3 + m_3)  \\
&\hspace{5em}+ \left( -w_1^{m}\beta^{+-+}_{w^k,m} + w_1^{k} \beta^{+-+}_{w^m,k} \right)
{\rm S}_1 (k_1 - m_1) {\rm C}_2(k_2 + m_2) {\rm S}_3 (k_3 - m_3)  \\
&\hspace{5em}+  \left( -w_1^{m}\beta^{+--}_{w^k,m} - w_1^{k} \beta^{+--}_{w^m,k} \right)
{\rm S}_1 (k_1 + m_1) {\rm C}_2(k_2 - m_2) {\rm S}_3 (k_3 - m_3)  \\
&\hspace{5em}+  \left( w_1^{m} \beta^{+--}_{w^k,m} - w_1^{k} \beta^{+--}_{w^m,k}\right)
{\rm S}_1 (k_1 - m_1) {\rm C}_2(k_2 + m_2) {\rm S}_3 (k_3 + m_3), 
\end{split}
\end{gather}
\begin{gather}
\begin{split} \label{2ndCoorMix}
&\left( \left( Y^k \cdot \nabla   \right) Z^m + \left( Z^m \cdot \nabla   \right) Y^k \right)_2\\
&\hspace{4em}= \left(w_2^{m}\beta^{+++}_{w^k,m} + w_2^{k} \beta^{+++}_{w^m,k} \right)
{\rm C}_1 (k_1 + m_1) {\rm S}_2(k_2 + m_2) {\rm S}_3 (k_3 + m_3)  \\
&\hspace{5em}+ \left( -w_2^{m}\beta^{+++}_{w^k,m} + w_2^{k} \beta^{+++}_{w^m,k}\right)
{\rm C}_1 (k_1 - m_1) {\rm S}_2(k_2 - m_2) {\rm S}_3 (k_3 - m_3)  \\
&\hspace{5em}+ \left( - w_2^{m}\beta^{++-}_{w^k,m} - w_2^{k} \beta^{++-}_{w^m,k} \right)
{\rm C}_1 (k_1 + m_1) {\rm S}_2(k_2 + m_2) {\rm S}_3 (k_3 - m_3)  \\
&\hspace{5em}+  \left( w_2^{m}\beta^{++-}_{w^k,m} - w_2^{k} \beta^{++-}_{w^m,k} \right)
{\rm C}_1 (k_1 - m_1) {\rm S}_2(k_2 - m_2) {\rm S}_3 (k_3 + m_3)  \\
&\hspace{5em}+ \left( -w_2^{m}\beta^{+-+}_{w^k,m} + w_2^{k} \beta^{+-+}_{w^m,k} \right)
{\rm C}_1 (k_1 + m_1) {\rm S}_2(k_2 - m_2) {\rm S}_3 (k_3 + m_3)  \\
&\hspace{5em}+ \left( w_2^{m}\beta^{+-+}_{w^k,m} + w_2^{k} \beta^{+-+}_{w^m,k} \right)
{\rm C}_1 (k_1 - m_1) {\rm S}_2(k_2 + m_2) {\rm S}_3 (k_3 - m_3)  \\
&\hspace{5em}+  \left( w_2^{m}\beta^{+--}_{w^k,m} - w_2^{k} \beta^{+--}_{w^m,k} \right)
{\rm C}_1 (k_1 + m_1) {\rm S}_2(k_2 - m_2) {\rm S}_3 (k_3 - m_3)  \\
&\hspace{5em}+  \left( -w_2^{m} \beta^{+--}_{w^k,m} - w_2^{k} \beta^{+--}_{w^m,k}\right)
{\rm C}_1 (k_1 - m_1) {\rm S}_2(k_2 + m_2) {\rm S}_3 (k_3 + m_3), 
\end{split}
\end{gather}
\begin{gather}
\begin{split} \label{3rdCoorMix}
&\left( \left( Y^k \cdot \nabla   \right) Z^m + \left( Z^m \cdot \nabla   \right) Y^k \right)_3\\
&\hspace{4em}= \left(-w_3^{m}\beta^{+++}_{w^k,m} - w_3^{k} \beta^{+++}_{w^m,k} \right)
{\rm C}_1 (k_1 + m_1) {\rm C}_2(k_2 + m_2) {\rm C}_3 (k_3 + m_3)  \\
&\hspace{5em}+ \left( w_3^{m}\beta^{+++}_{w^k,m} - w_3^{k} \beta^{+++}_{w^m,k}\right)
{\rm C}_1 (k_1 - m_1) {\rm C}_2(k_2 - m_2) {\rm C}_3 (k_3 - m_3)  \\
&\hspace{5em}+ \left( - w_3^{m}\beta^{++-}_{w^k,m} + w_3^{k} \beta^{++-}_{w^m,k} \right)
{\rm C}_1 (k_1 + m_1) {\rm C}_2(k_2 + m_2) {\rm C}_3 (k_3 - m_3)  \\
&\hspace{5em}+  \left( w_3^{m}\beta^{++-}_{w^k,m} + w_3^{k} \beta^{++-}_{w^m,k} \right)
{\rm C}_1 (k_1 - m_1) {\rm C}_2(k_2 - m_2) {\rm C}_3 (k_3 + m_3)  \\
&\hspace{5em}+ \left( -w_3^{m}\beta^{+-+}_{w^k,m} - w_3^{k} \beta^{+-+}_{w^m,k} \right)
{\rm C}_1 (k_1 + m_1) {\rm C}_2(k_2 - m_2) {\rm C}_3 (k_3 + m_3)  \\
&\hspace{5em}+ \left( w_3^{m}\beta^{+-+}_{w^k,m} - w_3^{k} \beta^{+-+}_{w^m,k} \right)
{\rm C}_1 (k_1 - m_1) {\rm C}_2(k_2 + m_2) {\rm C}_3 (k_3 - m_3)  \\
&\hspace{5em}+  \left( -w_3^{m}\beta^{+--}_{w^k,m} + w_3^{k} \beta^{+--}_{w^m,k} \right)
{\rm C}_1 (k_1 + m_1) {\rm C}_2(k_2 - m_2) {\rm C}_3 (k_3 - m_3)  \\
&\hspace{5em}+  \left( w_3^{m} \beta^{+--}_{w^k,m} + w_3^{k} \beta^{+--}_{w^m,k}\right)
{\rm C}_1 (k_1 - m_1) {\rm C}_2(k_2 + m_2) {\rm C}_3 (k_3 + m_3).
\end{split}
\end{gather}
\end{subequations}
\begin{remark} 
We do not present the coordinates of the vector \\$\left( Z^k \cdot \nabla   \right) Z^m + \left( Z^m \cdot \nabla   \right) Z^k$
because we will not need them in the construction of $Z^{j(n),n}$. The vectors generate functions of the type $Y^{j(n),n}$.
\end{remark}

\subsection{Two fruitful lemmas}
Next we introduce two fruitful lemmas which play the main role in the proof below. These lemmas help us to avoid complicated computations to obtain an explicit form of the operator $\BB(a,b)$ as some works before in 2D cases (see \cite{Rod-Sev05,PhanRod-ecc15,Rod06,Rod-Thesis08}).

\begin{lemma} \label{lemma-lin-indp_Z}
Let us be given $\alpha, \gamma, \in \R^3$ and $k \in \N_0^3$. Then 
\[\linspan\{ \Pi \ZZ_\alpha^k,   \Pi \ZZ_\gamma^k\} =\linspan Z^{\{1,2\},k}\]
if, and only if, the family~$\{\alpha,\gamma,k\}$ is linearly independent. 
\end{lemma}
In Lemma~\ref{lemma-lin-indp_Z}, we denote
\begin{align*}
\ZZ^n\coloneqq\begin{pmatrix}
z_1 \sin \left( \frac{k_1 \pi x_1}{L_1} \right) \cos \left( \frac{k_2 \pi x_2}{L_2} \right) \sin \left( \frac{k_3 \pi x_3}{L_3} \right)\\
z_2 \cos \left( \frac{k_1 \pi x_1}{L_1} \right) \sin \left( \frac{k_2 \pi x_2}{L_2} \right) \sin \left( \frac{k_3 \pi x_3}{L_3} \right)\\ 
z_3 \cos \left( \frac{k_1 \pi x_1}{L_1} \right) \cos \left( \frac{k_2 \pi x_2}{L_2} \right) \cos \left( \frac{k_3 \pi x_3}{L_3} \right),
 \end{pmatrix}
\end{align*}
for given~$z\in\R^3$ and $n\in\N^3$. The proof is analogous to the Lemma 3.1 in \cite{PhanRod18JDCS}. We will also need the following
relaxed version.
 \begin{lemma} \label{lemma-lin-indp2}
Let us be given $\alpha, \gamma, \delta \in \R^3$ and $k \in \N_0^3$. Then 
\[\linspan\{ \Pi \ZZ_\alpha^k,   \Pi \ZZ_\gamma^k, \Pi \ZZ_\delta^k\} =\linspan Z^{\{1,2\},k}\]
if, and only if, at least one of the families~$\{\alpha,\gamma,k\}$, $\{\alpha,\delta,k\}$, and~$\{\gamma,\delta,k\}$ is linearly independent. 
\end{lemma}
\begin{proof}
Notice that the inclusion~$\linspan\{ \Pi \ZZ_\alpha^k,   \Pi \ZZ_\gamma^k, \Pi \ZZ_\delta^k\} \subseteq\linspan Z^{\{1,2\},k}$ holds true
for all~$\alpha, \gamma, \delta \in \R^3$.
Then, the reverse inclusion holds if, and only if, we can set two vectors in~$\{ \Pi \ZZ_\alpha^k,   \Pi \ZZ_\gamma^k, \Pi \ZZ_\delta^k\}$ which are linearly independent.
The Lemma follows straightforwardly from Lemma~\ref{lemma-lin-indp_Z}. 
\end{proof}
\begin{remark}
Two lemmas above will help us to prove the linear independent property. However mostly in the proof (for example in Step \ref{Step352}, equation \eqref{eq-lem-fail}), we can not prove the linear independence by using only Lemma \ref{lemma-lin-indp_Z} (as the procedure that we used before in the case of 3D Rectangle in \cite{PhanRod18JDCS}). Therefore we must introduce the Lemma \ref{lemma-lin-indp2} and apply it for these situations. The main difference between two Lemmas is that we need three choices to generate the vectors $Z^k$ in Lemma \ref{lemma-lin-indp2} comparing with two choices in Lemma \ref{lemma-lin-indp_Z}. 
\end{remark}

\subsection{Proof of all Lemmas \ref{lem-genR3}--\ref{lem-genq+1}}
\label{sec: ProofSat3DCyl}
We introduce notations used frequently in proofs below. For $z_1, z_2, z_3 \in \R^3$, the matrix $\left( z_1 \mid z_2 \mid z_3 \right)$ is a square matrix whose each column is the vector $z_i$ for $i \in \{1,2,3\}$; and $\det\left( z_1 \mid z_2 \mid z_3 \right)$ denotes its determinant.
 
 \bigskip
\subsubsection{Proof of Lemma \ref{lem-genR3}}
\label{subsec-proof.genR3}
In this proof, we will construct all eigenfunctions $Z^{j(n),n}$ with $n \in \RR^{q+1}_3$. It is equivalent that the third coordinates of indices $n$ are always $p+1$. 
We divide this proof into three steps
\begin{enumerate}
\renewcommand{\theenumi}{{\sf3.3.\arabic{enumi}}} 
 \renewcommand{\labelenumi}{} 
\item $\bullet$~Step~\theenumi:\label{Step351}{ Generating $Z^{1,n}$ with $n \in \{ (0,l, q+1),~(l,0, q+1) \mid 1 \le l \leq q \}$.}
\item $\bullet$~Step~\theenumi:\label{Step352}{ Generating $Z^{\{1,2\},n}$  with  $n \in \{ (1,l, q+1),~(l,1, q+1) \mid 1 \le l \le q \}$.}
\item $\bullet$~Step~\theenumi:\label{Step353}{ Generating all $Z^{\{1,2\},n}$ with $n \in \{ (n_1,n_2, q+1) \mid  2 \le n_1 \le q, 2 \le n_2 \le q \}$ by induction.}
\end{enumerate}

\bigskip
\textbf{Step \ref{Step351}:} Generating $Z^{1,n}$ with $n \in \{ (0,l, q+1),~(l,0, q+1) \mid 1 \le l \leq q \}$.

In this step, we will generate eigenfunctions with indices $n$ whose first and second coordinates belongs two orthogonal sides of the square $[0,q] \times [0,q]$ in Figure \ref{fig_step1}. Notice that the index $(0,0,q+1)$ does not correspond with any eigenfunction. 
  
\begin{figure}[ht] \centering 
\includegraphics[width=0.5\textwidth]{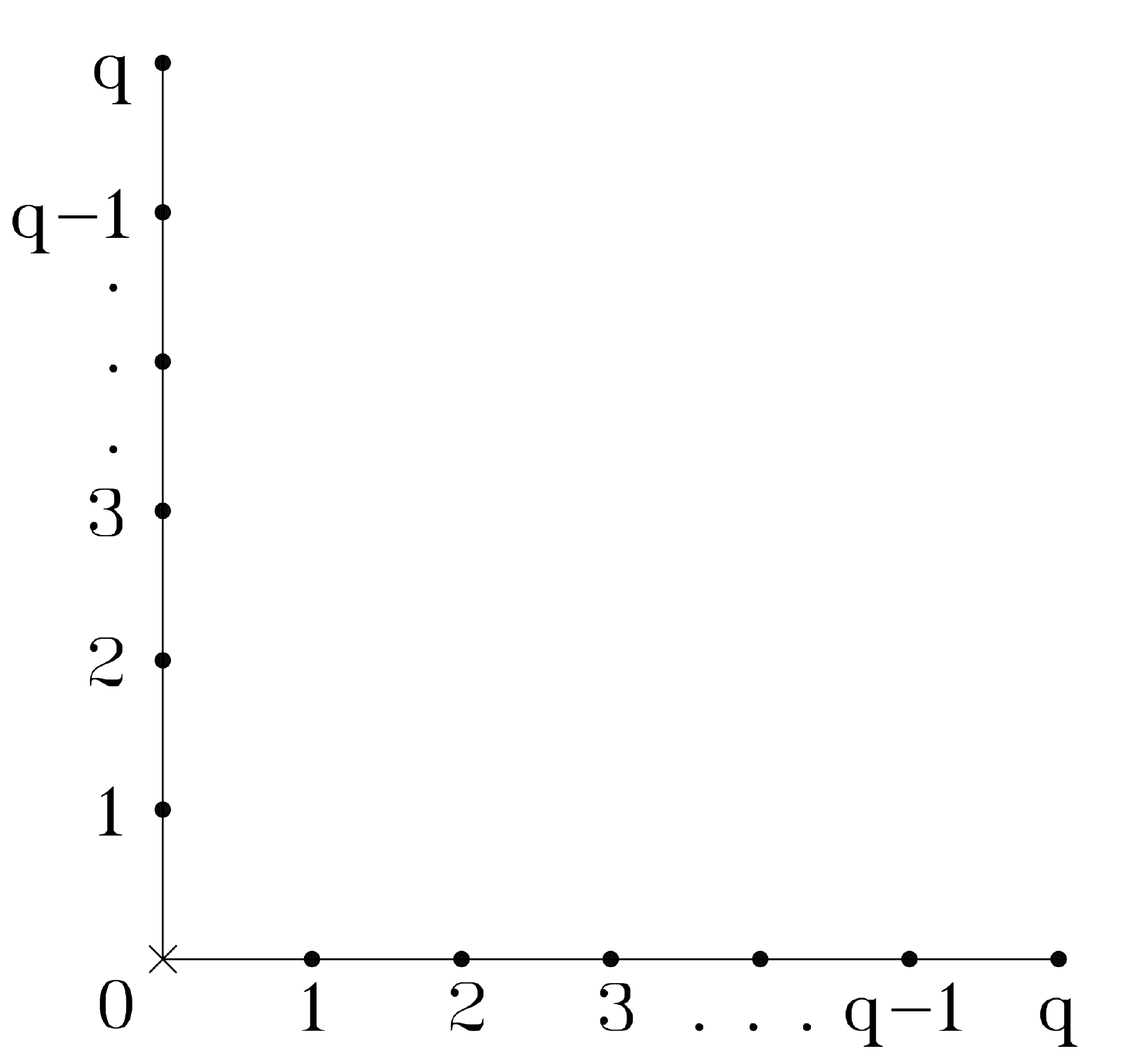}
\caption{Step \ref{Step351}.}
\label{fig_step1}
\end{figure}

{\em The case $n=(0,l,q+1)$.} We may follow the result in 2D Cylinder addressed in~\cite{PhanRod-ecc15}.
Indeed from~\eqref{CoorYZ_Cyl} we find that for~$k$ and~$m$ such that~$k_1=m_1=0$, 
\[
Y^k=\begin{pmatrix}
     0\\\widehat Y^k
    \end{pmatrix}\qquad\mbox{and}\qquad Z^k=\begin{pmatrix}
     0\\\widehat Z^k
    \end{pmatrix}
\]
where for suitable constants~$C_1$ and~$C_2$
\[
 \widehat Y^k=C_1\begin{pmatrix}
     \frac{-k_3\pi}{L_3}\sin(\frac{k_2\pi x_2}{L_2})\cos(\frac{k_3\pi x_3}{L_3})\\ \frac{k_2\pi}{L_2}\cos(\frac{k_2\pi x_2}{L_2})\sin(\frac{k_3\pi x_3}{L_3})
    \end{pmatrix}~~\mbox{and}~~\widehat Z^k=C_1\begin{pmatrix}
     \frac{k_3\pi}{L_3}\sin(\frac{k_2\pi x_2}{L_2})\sin(\frac{k_3\pi x_3}{L_3})\\ \frac{k_2\pi}{L_2}\cos(\frac{k_2\pi x_2}{L_2})\cos(\frac{k_3\pi x_3}{L_3})
    \end{pmatrix},
\]
where the functions~$\widehat Y^k$ and~$\widehat Z^k$ are eigenfuntions of the Stokes operator in~${\rm C}_2=(0,L_2)\times \frac{L_3}{\pi}\mathbb{S}^1$ under Lions boundary conditions.
Using an argument that is similar to the one used to derive~$\linspan\{Y^n\mid n\in\sS^{q+1}_{\rm R}, \#_0(n)=1\}\subset \GG^{q}$ as in \cite[Section 3.4]{PhanRod18JDCS}, we can derive that
$$
 \BB(Y^k)Z^k=\begin{pmatrix}
              0\\\Pi_2\left( \left( \widehat Y^k \cdot \nabla_2   \right) \widehat Z^m + \left( \widehat Z^m \cdot \nabla_2   \right) \widehat Y^k \right)
             \end{pmatrix}
$$
with~$\Pi_2$ being the orthogonal projection onto $H_2=\{u\in L^2({\rm C}_2,T{\rm C}_2)\mid u\cdot\nnn=0, \diver_2 u=0\}$
and~$\nabla_2$ and~$\diver_2$ being the gradient and divergence operators in~${\rm C}_2\sim(0,L_2)\times(0,2L_3)$.

Therefore from~\cite[proof of Theorem~4.1]{PhanRod-ecc15} we know that
\begin{subequations}\label{step351.res}
 \begin{align}
 &\{Z^{1,n}\mid n=(0,l,q+1), 0<l\le q\}\subset\GG^{q}.
 \intertext{ and a similar argument gives us}
 &\{Z^{1,n}\mid n=(l,0,q+1), 0<l\le q\}\subset\GG^{q}.
\end{align}
\end{subequations}

\bigskip
\textbf{Step \ref{Step352}:} Generating $Z^{\{1,2\},n}$ with $n \in \{ (1,l, q+1),~(l,1, q+1) \mid 1 \le l \le q \}$.

After Step \ref{Step351}, we obtain eigenfunctions with indices $n$ whose first and second coordinates belongs two orthogonal sides of the square. In this step, we will continue with two new lines plotted in Figure \ref{fig_step2}.

\begin{figure}[ht] \centering 
\includegraphics[width=0.5\textwidth]{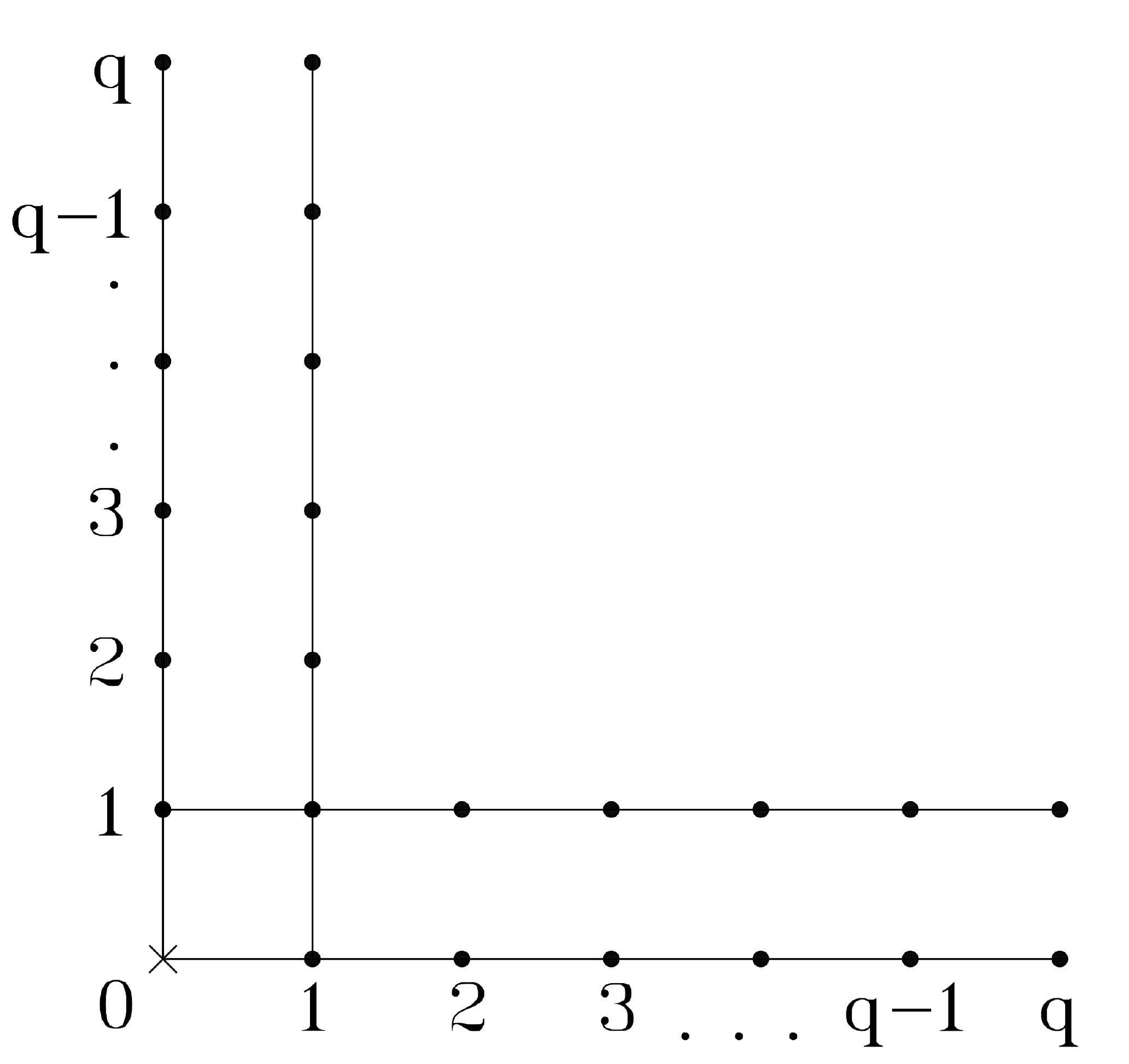}
\caption{Step \ref{Step352}.}
\label{fig_step2}
\end{figure}

{\em The case $n = (1,1,q+1)$}. Firstly, we choose
\begin{align*}
k &=(1,0,q),& m &=(0,1,1),\\ w^k &= (L_1q,0,-L_3),& w^m &= (0,L_2,-L_3).
\end{align*}
From \eqref{CoorYZ_Cyl}, this choice gives us
\begin{align*}
\left( Y^k \cdot \nabla   \right) Z^m + \left( Z^m \cdot \nabla   \right) Y^k = \ZZ^{(1,1,q+1)}_{z_{\alpha^1}} + \ZZ^{(1,1,q-1)}_{z_{\alpha^2}}, 
\end{align*}
for suitable $z_{\alpha^1}, z_{\alpha^2} \in \R^3$.
By the induction hypothesis~\eqref{Sat_Inc_CZ}, we have $Z^{\{1,2 \},(1,1,q-1)} \in \GG^{q-1}\subset \GG^{q} $. It is equivalent that $\ZZ^{(1,1,q-1)}_{z_{\alpha^2}} \in \GG^{q}$. Next, from
\begin{align*}
\beta_{w^k,m}^{\star_1\star_2+} = -\frac{\pi}{8},\quad\beta_{w^k,m}^{\star_1\star_2-} = \frac{\pi}{8},\quad\beta_{w^m,k}^{\star_1\star_2+} = -\frac{\pi}{8}q,
\quad\mbox{and}\quad\beta_{w^m,k}^{\star_1\star_2-} = \frac{\pi}{8}q,
\end{align*}
with $\{\star_1,\star_2\}\subseteq \{+,-\}$, following the expressions in~\eqref{CoorYZ_Cyl}, we find
\begin{align*}
 z_{\alpha^1}
&= \begin{pmatrix}
0 + {L_1}q\left( \beta_{w^m,k}^{+++} - \beta_{w^m,k}^{++-} + \beta_{w^m,k}^{+-+} - \beta_{w^m,k}^{+--} \right) \\
{L_2} \left( \beta_{w^k,m}^{+++} + \beta_{w^k,m}^{++-} \sign(0-1) - \beta_{w^k,m}^{+-+} \sign(0-1) -\beta_{w^k,m}^{+--} \right) + 0 \\
-{L_3} \left(- \beta_{w^k,m}^{+++}  - \beta_{w^k,m}^{+-+} + \beta_{w^k,m}^{+--} 
+ \beta_{w^k,m}^{++-} \right) - {L_3} \left( -\beta_{w^m,k}^{+++} - \beta_{w^m,k}^{+-+} + \beta_{w^m,k}^{+--} + \beta_{w^m,k}^{++-} \right) 
\end{pmatrix} \\&= -\frac{\pi}{2}
\begin{pmatrix}
L_1q^2 \\
{ L_2} \\
L_3(q+1)
\end{pmatrix},
\end{align*}
and we can conclude that $\Pi \ZZ^{(1,1,q+1)}_{z_{\alpha^1}} \in \GG^{q}$.

Secondly, we choose
\begin{align*}
k &=(0,1,q),& m &=(1,0,1),\\ w^k &= (0,L_2q, -L_3),& w^m &= (L_1,0,-L_3)
\end{align*}
which again gives us~$\left( Y^k \cdot \nabla   \right) Z^m + \left( Z^m \cdot \nabla   \right) Y^k = \ZZ^{(1,1,q+1)}_{z_{\gamma^1}} + \ZZ^{(1,1,q-1)}_{z_{\gamma^2}}$ with
\begin{align*}
z_{\gamma^1} = -\frac{\pi}{2}\begin{pmatrix}
 L_1 \\
 L_2q^2 \\
 L_3 (q+1) \\
\end{pmatrix}.
\end{align*}
Again, we get that $\Pi\ZZ^{(1,1,q+1)}_{z_{\gamma^1}} \in \GG^q$. 

We observe that
\begin{align*}\det\left(n\mid z_{\alpha^1}\mid z_{\gamma^1}\right)
&= \frac{\pi^2}{4} (q+1)
\det \begin{pmatrix}
1 & L_1 q^2 & L_1\\
1 & L_2 & L_2q^2\\
1 & L_3 & L_3
\end{pmatrix} \\
&= \frac{\pi^2}{4} (q+1)^2 (q-1) \left( L_1L_2 q^2 + L_1L_2 - L_2L_3 - L_1L_3 \right),
\end{align*}
and that, since $q\ge4$,
\begin{equation} \label{eq-lem-fail}
\det\left(n\mid z_{\alpha^1}\mid z_{\gamma^1}\right)  = 0 \quad \Longleftrightarrow\quad q^2 = \frac{L_2L_3+L_1L_3-L_1L_2}{L_1L_2}. 
\end{equation}
Thus from Lemma~\ref{lemma-lin-indp_Z} we conclude that~$ \Pi \ZZ^{(1,1,q+1)}_{z_{\alpha^1}}$ and $\Pi \ZZ^{(1,1,q+1)}_{z_{\gamma^1}}$
are not necessarily linearly independent. So, next
we want to use Lemma~\ref{lemma-lin-indp2}. For that, we choose the third quadruple
\begin{align*}
k &=(1,0,q-1),& m &=(0,1,2),\\ w^k &= (L_1(q-1),0,-L_3),& w^m &= (0,2L_2,-L_3),
\end{align*}
which gives us 
\begin{align*}
\left( Y^k \cdot \nabla   \right) Z^m + \left( Y^m \cdot \nabla   \right) Z^k = \ZZ^{(1,1,q+1)}_{z_{\delta^1}} + \ZZ^{(1,1,q-3)}_{z_{\delta^2}}, 
\end{align*}
for suitable $z_{\delta^1}, z_{\delta^2} \in \R^3$.
Since by~\eqref{Sat_Inc_CZ} $Z^{\{1,2\},(1,1,q-3)}  \in \GG^{q-1} \subset \GG^{q}$, we can conclude that 
 $\Pi \ZZ^{(1,1,q+1)}_{z_{\delta^1}} \in \GG^{q}$. We can also find
\begin{align*}
&~z_{\delta^1}=  -\frac{\pi}{2}
\begin{pmatrix}
L_1(q-1)^2 \\
4 L_2  \\
 L_3(q+1)
\end{pmatrix}.
\end{align*}
Now, we compute
\begin{align*}
&\frac{4}{\pi^2(q+1)}\det\left(n \mid z_{\alpha^1}\mid z_{\delta^1} \right)  =   \det \begin{pmatrix}
1 & L_1 q^2 & L_1 (q-1)^2\\
1 & L_2 & 4L_2\\
1 & L_3 & L_3 
\end{pmatrix} \\ 
&\qquad=    3L_1L_2 q^2 + 2 (L_1L_2 - L_1L_3) q + L_1L_3 - L_1L_2 - 3L_2L_3  \\ 
&\frac{4}{\pi^2(q+1)}\det \left(n \mid z_{\gamma^1} \mid z_{\delta^1} \right)  =    \det \begin{pmatrix}
1 & L_1  & L_1 (q-1)^2\\
1 & L_2 q^2 & 4L_2\\
1 & L_3 & L_3 
\end{pmatrix} \\ 
&\qquad=   -L_1L_2 q^4 + 2L_1L_2 q^3 + (L_2L_3+ L_1L_3 - L_1L_2) q^2 -2L_1L_3q + 4L_1L_2 - 4L_2L_3.  
\end{align*}
From which we conclude that 
$$\det\left(n\mid z_{\alpha^1}\mid z_{\gamma^1}\right) = \det\left(n \mid z_{\alpha^1}\mid z_{\delta^1} \right)= \det \left(n \mid z_{\gamma^1} \mid z_{\delta^1} \right) = 0$$ if, and only if,  
\begin{subequations}
\begin{align}
L_1L_2 q^2 + L_1L_2 - L_2L_3 - L_1L_3 &= 0 \label{step12-eq1} \\
 L_2 ( L_3- L_1) (q-2) &= 0 \label{step12-eq2}\\
 L_1 (L_2-L_3) (q-2) &= 0. \label{step12-eq3} 
\end{align}
\end{subequations}
Since $q \ge 4$ and $L_1,\,L_2,\,L_3 >0$, from \ref{step12-eq2} and \ref{step12-eq3}, it arrives that $L_1 = L_2 = L_3$. In this case, from \ref{step12-eq1} we arrive to
the contradiction $q = 1$. Hence, at least one of the families~$\{n,z_{\alpha^1},z_{\gamma^1}\}$, $\{n,z_{\alpha^1},z_{\delta^1}\}$, or~$\{n,z_{\gamma^1},z_{\delta^1}\}$ is
linearly independent. From Lemma~\ref{lemma-lin-indp2}
we can conclude that  
\begin{align} \label{res-step21-3Dcyl}
Z^{\{1,2\},(1,1,q+1)} \in \GG^{q}.
\end{align}

\medskip
{\em The case $n = (1, l , q +1 ) $ with $2 \le l \le q$.}\\
We will prove by induction. Assume that
\begin{equation}\label{IH-C1l}
 Z^{j(k),k} \subset  \GG^{q} \quad \text{with} \quad k = (1,l-2,q+1).
\end{equation}
We will prove that $Z^{\{1,2 \},(1,l, q+1)}$. 

Firstly, we choose
\begin{align*}
k &=(1,l-1,q),& m &=(0,1,1),\\ w^k &= (0,L_2q,L_3(1-l)),& w^m &= (0,L_2,-L_3),
\end{align*} 
This choice gives us 
\begin{align*}
\left( Y^k \cdot \nabla   \right) Z^m + \left( Z^m \cdot \nabla   \right) Y^k = \ZZ^{(1,l,q+1)}_{z_{\alpha^1}} + \ZZ^{(1,l-2,q+1)}_{z_{\alpha^2}} + \ZZ^{(1,l,q-1)}_{z_{\alpha^3}} +  \ZZ^{(1,l-2,q-1)}_{z_{\alpha^4}}.
\end{align*}
From~\eqref{IH.C} and~\eqref{IH-C1l}, we have that $\ZZ^{(1,l-2,q+1)}_{z_{\alpha^2}}$, $\ZZ^{(1,l,q-1)}_{z_{\alpha^3}}$, and  $\ZZ^{(1,l-2,q-1)}_{z_{\alpha^4}}$ belong in $\GG^{q}$. Therefore, we can conclude that $\Pi\ZZ^{(1,l,q+1)}_{z_{\alpha^1}} \in \GG^{q}$.\\
Next, from
\begin{align*}
\beta_{w^k,m}^{+++} =\beta_{w^k,m}^{-++} =  \frac{\pi (q-l +1)}{8}\quad\mbox{and}\quad\beta_{w^m,k}^{+++} =\beta_{w^m,k}^{-++} =  \frac{\pi(l-q-1)}{8},
\end{align*}
we obtain
\begin{align*}
z_{\alpha^1}
&= \begin{pmatrix}
0  \\
L_2 \left( \beta_{w^k,m}^{+++} + \beta_{w^k,m}^{-++} \right) + L_2q \left( \beta_{w^m,k}^{+++} + \beta_{w^m,k}^{-++} \right) \\
-L_3\left( -\beta_{w^k,m}^{+++} - \beta_{w^k,m}^{-++} \right) + L_3(1-l) \left( -\beta_{w^m,k}^{+++} - \beta_{w^m,k}^{-++} \right)
\end{pmatrix}\\
&=  -\frac{\pi}{4}
\begin{pmatrix}
0 \\
 L_2(q-l+1)(q-1)\\
 L_3 (q-l+1)(l-2)
\end{pmatrix}.
\end{align*} 

Secondly we choose
\begin{align*}
k &=(1,l-1,q),& m &=(1,1,1),\\ w^k &= (L_1q,0,-L_3),& w^m &= (0,L_2,-L_3),
\end{align*}
which allow us to obtain $\Pi\ZZ^{(1,l,q+1)}_{z_{\gamma^1}} \in \GG^{q}$, and from
\[
\beta_{w^k,m}^{+++} = \beta_{w^k,m}^{-++} = -\frac{\pi}{8}\quad\mbox{and}\quad\beta_{w^m,k}^{+++} = \beta_{w^m,k}^{-++} =  \frac{\pi}{8} (l-q-1), 
\]
we can find  
\begin{align*}
z_{\gamma^1}=  -\frac{\pi}{4}
\begin{pmatrix}
 L_1 q(q-l+1) \\
 L_2 \\
L_3(q-l+2)
\end{pmatrix}.
\end{align*}

Thirdly, we choose
\begin{align*}
k &=(1,l-1,q),& m &=(1,1,1),\\ w^k &= (0,L_2q,L_3(1-l)),& w^m &= (L_1,-L_2,0),
\end{align*}
which gives us that $\Pi\ZZ^{(1,l,q+1)}_{z_{\delta^1}} \in \GG^{q}$, with 
\begin{align*}
z_{\delta^1}= -\frac{\pi}{4}
\begin{pmatrix}
- L_1 (q-l+1) \\
 L_2 (ql-l+1) \\
L_3(l-1)^2
\end{pmatrix}.
\end{align*}

Next we compute 
\begin{align*}
\det \left(n\mid z_{\alpha^1}\mid z_{\gamma^1}\right)  &= -\frac{\pi^2}{16} q (q-l+1)^2 \left( L_1L_2 q^2 - L_1L_2 - L_2L_3 - L_1L_3 l(l-2)\right), \\
\det \left(n \mid z_{\alpha^1}\mid z_{\delta^1}\right)  &=\frac{\pi^2}{16} (q-l+1)^2 \left( L_1L_2 q^2 + L_2L_3 - L_1L_2-L_1L_3 l(l-2)\right) , 
\end{align*}
and observe that $\det \left(n\mid z_{\alpha^1}\mid z_{\gamma^1}\right)=\det \left(n \mid z_{\alpha^1}\mid z_{\delta^1}\right)=0$ if, and only if,
\begin{align*}
 L_1L_2 q^2 - L_1L_2 - L_2L_3 - L_1L_3 l(l-2) = L_1L_2 q^2 + L_2L_3 - L_1L_2-L_1L_3 l(l-2) = 0,
\end{align*}
because $2 \le l \le q$, which implies 
$2 L_2L_3 = 0$. This contradicts the fact that $L_1,\,L_2,\,L_3 >0$. Therefore
one of the families~$\{n,z_{\alpha^1},z_{\gamma^1}\}$ or~$\{n,z_{\alpha^1},z_{\delta^1}\}$ is linearly independent and, by Lemma~\ref{lemma-lin-indp2}, it follows
that $Z^{\{1,2\},(1,l,q+1)} \in \GG^{q}$. 

By induction, using~\eqref{step351.res},~\eqref{res-step21-3Dcyl}, and the induction hypothesis~\eqref{IH-C1l} it follows that
\begin{subequations}\label{step352.res}
\begin{align}
Z^{(1,l,q+1)} \in \GG^{q}, \qquad\mbox{for all}\quad 0\le l\le q\\
\intertext{and proceeding similarly we can also derive that}
Z^{(l,1,q+1)} \in \GG^{q},\qquad\mbox{for all}\quad 0\le l\le q.
\end{align}
\end{subequations}

\bigskip
\textbf{Step \ref{Step353}:} Generating the family $Z^{\{1,2\},n}$ with $n = (n_1,n_2, q+1)$ where $2 \le n_1 \le q $  and $2 \le n_2 \le q$.

In the final step, we will generate all remaining eigenfunctions with indices $n$ whose the first and second coordinates belongs to the square. The process is based on the induction \eqref{IH-Cn1n2}. 
\begin{figure}[ht] \centering 
\includegraphics[width=.5\textwidth]{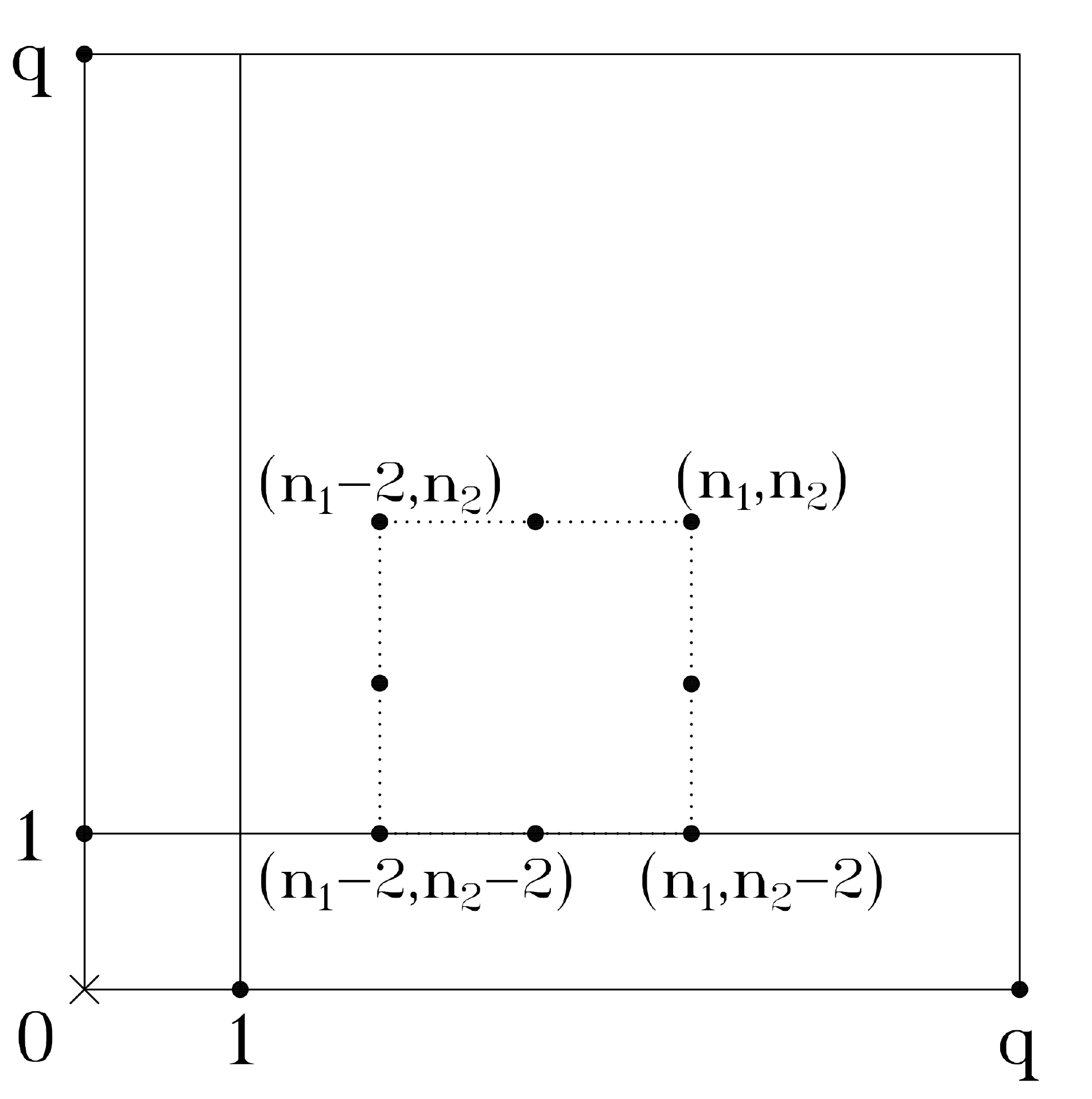}
\caption{Step \ref{Step353}.}
\label{fig_step3}
\end{figure}

Firstly, we introduce an induction hypothesis. Assume that
\addtocounter{equation}{1}
\begin{align}\label{IH-Cn1n2}
 &Z^{j(\kappa),\kappa} \in \GG^{q},\\
 &\mbox{for}\quad \kappa\in\{(n_1-2,n_2-2,q+1),(n_1-2,n_2,q+1),(n_1,n_2-2,q+1)\}.\notag
\end{align}
We will prove that $Z^{\{1,2\},(n_1,n_2,q+1)} \in \GG^q$.

By choosing
\begin{align*}
k &=(n_1-1,n_2-1,q),& m &=(1,1,1),\\ w^k &= (0,L_2q,L_3(1-n_2)),& w^m &= (0,L_2,-L_3),
\end{align*}
we obtain
 \begin{align*}
\left( Y^k \cdot \nabla   \right) Z^m + \left( Z^m \cdot \nabla   \right) Y^k = \ZZ^{(n_1,n_2,q+1)}_{z_{\alpha^1}} + {\sum_{i=2}^8{\ZZ^{\kappa^i}_{z_{\alpha^i}}},}
\end{align*}
with $\kappa^i \in \{(n_1-2,n_2-2,q-1), (n_1,n_2-2,q-1), (n_1-2,n_2,q-1),(n_1,n_2,q-1) (n_1-2,n_2-2,q+1), (n_1,n_2-2,q+1), (n_1-2,n_2,q+1) \}$.
Using the inductive hypothesis~\eqref{IH.C}, we find that 
$
Z^{j(\kappa), \kappa} \in \GG^{q},
$
for $\kappa \in \{ (n_1-2,n_2-2,q-1), (n_1,n_2-2,q-1), (n_1-2,n_2,q-1),(n_1,n_2,q-1) \}$. From the inductive hypothesis~\eqref{IH-Cn1n2} we also have
$
Z^{j(\kappa), \kappa} \in \GG^{q},
$ for $\kappa \in \{ (n_1-2,n_2-2,q+1), (n_1,n_2-2,q+1), (n_1-2,n_2,q+1) \}$.
Thus, we can conclude that $\Pi \ZZ^{(n_1,n_2,q+1)}_{z_{\alpha^1}} \in \GG^{q}$. 

Next, from
\[
 \beta_{w^k,m}^{+++} = \frac{\pi}{8} (q-n_2+1)\quad\mbox{and}\quad\beta_{w^m,k}^{+++} = \frac{\pi}{8} (n_2-q-1),
\]
we obtain
\begin{align*}
z_{\alpha^1} = -\frac{\pi}{8}
\begin{pmatrix}
0 \\
L_2(q-n_2+1)(q-1) \\
L_3(q-n_2+1)(n_2-2)
\end{pmatrix}.
\end{align*}
A second choice is
\begin{align*}
k &=(n_1-1,n_2-1,q),& m &=(1,1,1),\\ w^k &= (L_1q,0,L_3(1-n_1)),& w^m &= (0,L_2,-L_3),
\end{align*}
which gives us $\Pi \ZZ^{(n_1,n_2,q+1)}_{z_{\gamma^1}} \in \GG^{q}$. From
\[
\beta_{w^k,m}^{+++}  = \frac{1}{8} \left( q-n_1 +1 \right),\quad\mbox{and}\quad\beta_{w^m,k}^{+++} =  \frac{1}{8} (n_2 - q-1),
\] we obtain 
\begin{align*}
z_{\gamma^1}= \frac{\pi}{8}
\begin{pmatrix}
-L_1q(q-n_2+1) \\
 L_2(q-n_1+1) \\
L_3(1-n_1)(q-n_2+1) + L_3(q-n_1+1)
\end{pmatrix}.
\end{align*}
Another choice is
\begin{align*}
k &=(n_1,n_2-1,q),& m &=(0,1,1),\\ w^k &= (L_1(n_2-1),-L_2n_1,0),& w^m &= (0,L_2,-L_3),
\end{align*}
which gives us $\Pi \ZZ^{(n_1,n_2,q+1)}_{z_{\delta^1}} \in \GG^{q}$, with 
\begin{align*}
z_{\delta^1}= \frac{\pi}{8}
\begin{pmatrix}
 -L_1(q-n_2+1)(n_2-1) \\
 L_2n_1 (q-n_2) \\
 -L_3n_1
\end{pmatrix}.
\end{align*}

Next, from $2 \le n_1,\,n_2 \le q$ and
\begin{align*}
&-\frac{64}{\pi^2(q-n_2+1)}\det \left(n\mid z_{\alpha^1}\mid z_{\gamma^1}\right)\\
&\qquad=  \det \begin{pmatrix}
n_1  & 0 & -L_1(q-n_2+1)(n_2-1)\\
n_2 & L_2(q-1)&L_2(q-n_1+1) \\
q+1 & L_3(n_2-2) &L_3(1-n_1)(q-n_2+1) + L_3(q-n_1+1)
\end{pmatrix} \\
&\qquad= q(q-n_2+1) \left( L_1L_2 q^2 - L_1L_2 - L_2L_3n_1 (n_1 - 2) -  L_1L_3n_2 (n_2 - 2)  \right), \\
&-\frac{64}{\pi^2(q-n_2+1)}\det \left(n\mid z_{\alpha^1}\mid z_{\delta^1}\right)=  \det \begin{pmatrix}
n_1  & 0 & -L_1q(q-n_2+1)\\
n_2 & L_2(q-1)&L_2n_1 (q-n_2) \\
q+1 & L_3(n_2-2) & -L_3n_1
\end{pmatrix} \\ 
&\qquad=  (n_2 -1) (q-n_2 +1)  \left( L_1L_2 q^2 - L_1L_2 - L_2L_3n_1^2 -  L_1L_3n_2 (n_2 - 2)  \right),
\end{align*}
we have that $\det \left(n\mid z_{\alpha^1}\mid z_{\gamma^1}\right)=\det \left(n\mid z_{\alpha^1}\mid z_{\delta^1}\right)=0 $ if only if $2L_2L_3 n_1 = 0$. This contradicts the fact that~$L_2$, $L_3$, and~$n_1$ are
positive. Thus, one of the families~$\{n,z_{\alpha^1},z_{\gamma^1}\}$ or~$\{n,z_{\alpha^1},z_{\delta^1}\}$ is linearly independent. By Lemma~\ref{lemma-lin-indp2} it follows that
$Z^{\{1,2\},(n_1,n_2,q+1)}\in\GG^{q}$.

Using~\eqref{IH.C},~\eqref{step351.res},~\eqref{step352.res}, and the induction hypothesis~\eqref{IH-Cn1n2}, we conclude that $Z^{j(n),n} $
with $n = (n_1,n_2,q+1)$. Finally, we obtain 
\begin{align} \label{lem35.res}
Z^{j(n),n} \in \GG^{q}  \quad \text{for all~} n \in \RR^{q+1}_3.
\end{align} \qed

\subsubsection{Proof of Lemma \ref{lem-genR12}} \label{subsec-proof.genR12}

Notice that the
cases~$n \in  {\RR^{q+1}_1} $ and~$n \in  {\RR^{q+1}_2}$ are analogous. On the other hand, since we consider the periodicity assumption in
the third direction and Lions boundary conditions in the first two directions, these
cases must be addressed separately from the case~$n \in  {\RR^{q+1}_3}$ treated  in
section~\ref{subsec-proof.genR3}. Let us take~$n \in  {\RR^{q+1}_1}$. Again we divide this proof into three main steps
\begin{enumerate}
\renewcommand{\theenumi}{{\sf3.4.\arabic{enumi}}} 
 \renewcommand{\labelenumi}{} 
\item $\bullet$~Step~\theenumi:\label{Step361} Generating $Z^{1,n}$ with  $n \in \{ (q+1,l,0), (q+1,0,l) \mid 1 \le l\le q \}$.
\item $\bullet$~Step~\theenumi:\label{Step362} Generating $Z^{\{1,2\},n}$ with $n \in \{(q+1,l,1), (q+1,l,1) \mid 1 \le l \le q \}$.
\item $\bullet$~Step~\theenumi:\label{Step363} Generating $Z^{\{1,2\},n}$ with $n \in \{ (q+1,n_1,n_2) \mid 2 \le n_1,n_2 \le q \}$.
\end{enumerate}

\bigskip
\textbf{Step \ref{Step361}:} Generating $Z^{1,n}$ with $n \in \{ (q+1,l,0), (q+1,0,l) \mid 1 \le l\le q \}$.

{\em The case $n = (q+1,l,0)$}. We choose
\begin{align*}
k &=(q,l,0),& m &=(1,0,0),\\ w^k &= (L_1 l,-L_2 q, 0),& w^m &= (0,0,1),
\end{align*}
which gives us $\Pi \ZZ^{(q+1,l,0)}_{z_{\gamma^1}} \in \GG^{q}$. From 
$
\beta_{w^k,m}^{+\star_1\star_2} = \frac{\pi}{8}l,\quad\beta_{w^m,k} = 0,
$
with $\{\star_1,\star_2\}\subseteq \{+,-\}$, we get  
\begin{align*}
z_{\gamma^1} = \frac{\pi}{2}  \begin{pmatrix}
0 \\ 0 \\  l   
\end{pmatrix}. 
\end{align*} 
Observe that $\ZZ^{1,(q+1,l,0)}_{z_{\gamma^1}}=\frac{\pi}{2} l\begin{pmatrix}
                                                                    0\\0\\\cos(\frac{(q+1)\pi x_1}{L_1})\cos(\frac{l \pi x_2}{L_2}) 
                                                                   \end{pmatrix}$, it means that for a suitable constant $\zeta \neq 0$, $Z^{1,(q+1,l,0)} = \zeta \ZZ^{1,(q+1,l,0)}_{z_{\gamma^1}} $.
Hence, we conclude that
\begin{equation}\label{step361.res1}
 Z^{1,(q+1,l,0)}\in\GG^{q},\quad\mbox{for all}\quad 1\le l\le q.
\end{equation}

To generate $Z^{1,n}$ with $n = (q+1,0,l)$, we can use the result for the 2D cylinder in~\cite{PhanRod-ecc15}. Notice that, for some constant~$\zeta\ne0$,
 \[
  Z^{1,(q+1,0,l)} = \zeta\begin{pmatrix}
 \frac{l \pi }{L_3}\sin \left( \frac{(q+1) \pi x_1}{L_1} \right)  \sin \left( \frac{l \pi x_3}{L_3} \right)\\
0\\ 
\frac{(q+1) \pi }{L_1} \cos \left( \frac{(q+1) \pi x_1}{L_1} \right)  \cos \left( \frac{l \pi x_3}{L_3} \right),
 \end{pmatrix}
\]
which is an eigenfunction of the Stokes operator in the 2D cylinder~$(0,L_1)\times\frac{L_2}{\pi}\mathbb{S}^1$, under Lions boundary conditions.
It follows, from~\cite[Theorem~4.1]{PhanRod-ecc15}, that
\begin{equation}\label{step361.res2}
 Z^{1,(q+1,0,l)}\in\GG^{q},\quad\mbox{for all}\quad 1\le l\le q.
\end{equation}

\bigskip
\textbf{Step \ref{Step362}:} Generating $Z^{\{1,2\},n}$ with $n \in \{(q+1,l,1), (q+1,l,1) \mid 1 \le l \le q \}$.

{\em The case $n = (q+1,1,1)$ }. We firstly choose
\begin{align*}
k &=(q,0,1),& m &=(1,1,0),\\ w^k &= (L_1,0,-L_3q),& w^m &= (L_1,-L_2,0),
\end{align*}
Then, by changing the roles of~$k$ and~$m$ in from~\eqref{CoorYZ_Cyl}, we obtain
\begin{align*}
\left( Z^k \cdot \nabla   \right) Y^m + \left( Y^m \cdot \nabla   \right) Z^k = \ZZ^{(q+1,1,1)}_{z_{\alpha^1}} + \ZZ^{(q-1,1,1)}_{z_{\alpha^2}}, 
\end{align*}
for suitable $z_{\alpha^1}, z_{\alpha^2} \in \R^3$. By~\eqref{IH.C}, we have $Z^{\{1,2 \},(q-1,1,1)} \in \GG^{q}$. Therefore we derive
that~$\Pi\ZZ^{(q+1,1,1)}_{z_{\alpha^1}}\in\GG^{q}$, and we can also find
\begin{align*}
z_{\alpha^1}= \frac{\pi}{2}
\begin{pmatrix}
L_1 (q+1)\\
- L_2\\
 L_3q^2
\end{pmatrix}.
\end{align*}

Secondly, we compute~$\left( Z^m \cdot \nabla   \right) Y^k + \left( Y^k \cdot \nabla   \right) Z^m$ with the choice
\begin{align*}
k &=(q,1,1),& m &=(1,0,0),\\ w^k &= (0,L_2,-L_3),& w^m &= (0,0,L_3),
\end{align*}
Analogously, we obtain that $\Pi \ZZ^{(q+1,1,1)}_{z^{\gamma^1}} \in \GG^{q}$, with 
\begin{align*}
z_{\gamma^1}= \frac{\pi}{2}
\begin{pmatrix}
0 \\
L_2 \\
L_3
\end{pmatrix}.
\end{align*}

Thirdly, we compute~$\left( Z^m \cdot \nabla   \right) Y^k + \left( Y^k \cdot \nabla   \right) Z^m$ with
\begin{align*}
k &=(q,1,1),& m &=(1,0,0),\\ w^k &= (L_1,0,-L_3q),& w^m &= (0,0,L_3),
\end{align*}
which gives us $\Pi \ZZ^{(q+1,1,1)}_{z^{\delta^1}} \in \GG^{q}$, with
\begin{align*}
z_{\delta^1}= \frac{\pi}{2}
\begin{pmatrix}
L_1  \\
0 \\
 L_3(q-1)
\end{pmatrix}.
\end{align*}

Now, observe that $\det\left( n \mid z_{\alpha^1}~\mid z_{\gamma^1} \right) = \det \left(n \mid z_{\gamma^1} \mid z_{\delta^1} \right)=0$ if and only if 
\begin{align*}
L_2L_3 q^2 +L_1L_3+ L_2L_3-L_1L_2 = L_2L_3 q^2 -L_2L_3+L_1L_3-L_1L_2 = 0,
\end{align*}
which implies the contradiction $2L_2L_3 = 0$, because $L_2,L_3 >0$. Therefore, by Lemma~\ref{lemma-lin-indp2},
\begin{equation} \label{step362.res1}
Z^{\{1,2\},(q+1,1,1)} \in \GG^{q}.
\end{equation}

\medskip
{\em The case $n=( q +1,1,l)$}. Let us introduce the induction hypothesis
\addtocounter{equation}{1}
\begin{equation}\label{IH-Cq1l}
 Z^{j(k),k} \in \GG^{q} ,\quad\mbox{if}\quad k=(q+1,1,l-2),\quad\mbox{for given }2\le l\le q.
\end{equation}
We prove that $ Z^{j(k),k} \in \GG^{q}$ with $k=(q+1,1,l)$.

\medskip
To generate $n = (q+1,1,l)$, firstly we compute $\left( Z^m \cdot \nabla   \right) Y^k + \left( Y^k \cdot \nabla   \right) Z^m$ with the choice
\begin{align*}
k &=(q,1,l),& m &=(1,0,0),\\ w^k &= (0,L_2l,-L_3),& w^m &= (0,0,L_3),
\end{align*}
which allow us to conclude that $\Pi \ZZ^{(q+1,1,l)}_{z^{\alpha^1}} \in \GG^{q}$ with 
\begin{align*}
z_{\alpha^1}= \frac{\pi}{2}
\begin{pmatrix}
0 \\
L_2 l^2  \\
L_3 l
\end{pmatrix}.
\end{align*}
Secondly, we compute $\left( Z^m \cdot \nabla   \right) Y^k + \left( Y^k \cdot \nabla   \right) Z^m$  with
\begin{align*}
k &=(q,1,l),& m &=(1,0,0),\\ w^k &= (L_1l,0,-L_3q),& w^m &= (0,0,L_3),
\end{align*}
which gives us $\Pi \ZZ^{(q+1,1,l)}_{z^{\gamma^1}} \in \GG^{q}$, with 
\begin{align*}
z_{\gamma^1}= \frac{\pi}{2}
\begin{pmatrix}
 L_1 l^2  \\
0 \\
L_3 l (q-1)
\end{pmatrix}.
\end{align*}
Thirdly  we compute $\left( Z^k \cdot \nabla   \right) Y^m + \left( Y^m \cdot \nabla   \right) Z^k$ with
\begin{align*}
k &=(q,0,l),& m &=(1,1,0),\\ w^k &= (L_1l,0,-L_3q),& w^m &= (L_1,-L_2,0),
\end{align*}
which gives us  $\Pi \ZZ^{(q+1,1,l)}_{z^{\delta^1}} \in \GG^{q}$, with 
\begin{align*}
z_{\delta^1}=  \frac{\pi}{2}
\begin{pmatrix}
 L_1 l (q+1)  \\
- L_2l \\
 L_3 q^2
\end{pmatrix}.
\end{align*}
Next we observe that $\det\left( n \mid z_{\alpha^1}~\mid z_{\gamma^1} \right) = \det\left( n \mid z_{\alpha^1}~\mid z_{\delta^1} \right)  =0 $ if, and only if,
\begin{align*}
l^3 (L_2L_3 q^2 - L_2L_3 + L_1L_3 - L_1L_2 l^2) = l^2 (q+1) (L_2L_3 q ^2 + L_2L_3 + L_1L_3 - L_1L_2l) = 0 
\end{align*}
which leads to the contradiction $0=2L_1L_2l(l-1) + 2 L_2L_3 \ge 2L_2(L_1 + L_3) > 0$, since $2 \le l \le q$. Then from Lemma~\ref{lemma-lin-indp2}
we conclude that
$
  \ZZ^{(q+1,1,l)}_{z^{\delta^1}} \in \GG^{q}. 
$
By induction, using~\eqref{step361.res1}, \eqref{step362.res1}, and the induction
hypothesis~\eqref{IH-Cq1l}, it follows that
\begin{equation}\label{step362.res2}
\ZZ^{(q+1,1,l)}_{z^{\delta^1}} \in \GG^{q},\quad\mbox{for all}\quad 1\le l\le q.
\end{equation}

\medskip
{\em The case $n=( q +1,l,1)$.} Let us introduce the induction hypothesis
\addtocounter{equation}{1}
\begin{equation}\label{IH-Cql1}
 Z^{j(k),k} \in \GG^{q} ,\quad\mbox{with}\quad k=(q+1,1,l-2),\quad l\ge2.
\end{equation}
We prove that $ Z^{j(k),k} \in \GG^{q}$ with $ k=(q+1,1,l)$.

To generate $n = (q+1,l,1)$, firstly we choose 
\begin{align*}
k &=(q,l,1),& m &=(1,0,0),\\ w^k &= (0,L_2,-L_3 l),& w^m &= (0,0,L_3),
\end{align*}
which allow us to conclude that $\Pi \ZZ^{(q+1,1,l)}_{z^{\alpha^1}} \in \GG^{q}$ with 
\begin{align*}
z_{\alpha^1}= \frac{\pi}{2}
\begin{pmatrix}
0 \\
L_2   \\
L_3 l
\end{pmatrix}.
\end{align*}
Secondly, we compute $\left( Z^m \cdot \nabla   \right) Y^k + \left( Y^k \cdot \nabla   \right) Z^m$  with
\begin{align*}
k &=(q,l,1),& m &=(1,0,0),\\ w^k &= (L_1,0,-L_3q),& w^m &= (0,0,L_3),
\end{align*}
which gives us $\Pi \ZZ^{(q+1,1,l)}_{z^{\gamma^1}} \in \GG^{q}$, with 
\begin{align*}
z_{\gamma^1}= \frac{\pi}{2}
\begin{pmatrix}
 L_1  \\
0 \\
L_3 l (q-1)
\end{pmatrix}.
\end{align*}
Thirdly  we compute $\left( Z^k \cdot \nabla   \right) Y^m + \left( Y^m \cdot \nabla   \right) Z^k$ with
\begin{align*}
k &=(q,l,0),& m &=(1,0,1),\\ w^k &= (L_1 l,-L_2 q, 0),& w^m &= (L_1,0, -L_3),
\end{align*}
which gives us  $\Pi \ZZ^{(q+1,l,1)}_{z^{\delta^1}} \in \GG^{q}$, with 
\begin{align*}
z_{\delta^1}=  \frac{\pi}{2}
\begin{pmatrix}
 L_1 l (q+1)  \\
- L_2 q^2 \\
 L_3 l
\end{pmatrix}.
\end{align*}
Next we observe that $\det\left( n \mid z_{\alpha^1} \mid z_{\gamma^1} \right) = \det\left( n \mid z_{\alpha^1} \mid z_{\delta^1} \right)  =0 $ if, and only if,
\begin{align*}
L_2 L_3 l q (q-1) + L_1 (L_3 l^2 - L_2) = L_2 L_3 q (q^2 + 1) + L_1 (L_3 l^2 - L_2) (q+1) = 0 
\end{align*}
which leads to the contradiction $0= (l-1) (q^2 -1) - 2 \ge q^2-3 > 0$, since $2 \le l \le q$. Then from Lemma~\ref{lemma-lin-indp2}
we conclude that
$
\ZZ^{(q+1,l,1)}_{z^{\delta^1}} \in \GG^{q}. 
$
By induction, using~\eqref{step361.res2}, \eqref{step362.res1}, and the induction
hypothesis~\eqref{IH-Cql1}, it follows that
\begin{align} \label{step362.res3}
Z^{\{1,2 \}(q+1,l,1)} \in \GG^{q} \quad \mbox{for all~} 1 \le l \le q.
\end{align}

\medskip
\textbf{Step \ref{Step363}:} Generating $Z^{\{1,2\},n}$ with $n \in \{ (q+1,n_1,n_2) \mid 2 \le n_1,n_2 \le q \}$.

\medskip
Let us introduce the inductive hypothesis
\begin{align}\label{IH-Cylqn1n2}
\begin{split}
&Z^{j(\kappa),\kappa} \in \GG^{q},\\ 
&\mbox{for}\quad \kappa\in \{ (q+1,n_1-2,n_2-2), (q+1,n_1,n_2-2), (q+1,n_1-2,n_2) \}.
\end{split}
\end{align}
We will prove that $Z^{j(\kappa),\kappa} \in \GG^{q}$ with $\kappa = (q+1,n_1,n_2)$.

\medskip   

To generate $n = (q+1, n_1,n_2)$, firstly we compute~$\left( Z^m \cdot \nabla   \right) Y^k + \left( Y^k \cdot \nabla   \right) Z^m$ with
\begin{align*}
k &=(q,n_1,n_2),& m &=(1,0,0),\\ w^k &= (0,L_2n_2,-L_3n_1),& w^m &= (0,0,L_3),
\end{align*}
which leads to $\Pi \ZZ^{(q+1,n_1,n_2)}_{z_{\alpha^1}} \in \GG^{q}$, with 
\begin{align*}
z_{\alpha^1}=  \frac{\pi}{2}
\begin{pmatrix}
0 \\
L_2 n_2^2  \\
L_3 n_1 n_2
\end{pmatrix}.
\end{align*}
Secondly we compute~$\left( Z^m \cdot \nabla   \right) Y^k + \left( Y^k \cdot \nabla   \right) Z^m$  with
\begin{align*}
k &=(q,n_1,n_2),& m &=(1,0,0),\\ w^k &= (L_1 n_2,0,-L_3q),& w^m &= (0,0,L_3),
\end{align*}
which leads us to $\Pi \ZZ^{(q+1,n_1,n_2)}_{z_{\gamma^1}} \in \GG^{q}$ where 
\begin{align*}
z_{\gamma^1}=  \frac{\pi}{2}
\begin{pmatrix}
 L_1 n_2^2  \\
0  \\
 L_3 n_2 (q-1)
\end{pmatrix}.
\end{align*}
Thirdly we compute~$\left( Z^k \cdot \nabla   \right) Y^m + \left( Y^m \cdot \nabla   \right) Z^k$ with
\begin{align*}
k &=(q,n_1-1,n_2),& m &=(1,1,0),\\ w^k &= (L_1(n_1-1),-L_2q,0),& w^m &= (L_1,-L_2,0),
\end{align*}
which gives us  $\Pi \ZZ^{(q+1,n_1,n_2)}_{z_{\delta^1}} \in \GG^{q}$ with 
\begin{align*}
z_{\delta^1}=  \frac{\pi}{4}
\begin{pmatrix}
 L_1 (q-n_1+1)(n_1-2) \\
  L_2 (q-n_1+1)(2-q) \\
0
\end{pmatrix}.
\end{align*}
Now $\det\left( n \mid z_{\alpha^1} \mid z_{\gamma^1} \right) = \det\left( n \mid z_{\alpha^1} \mid z_{\delta^1} \right)  =0  $ if, and only if, 
\begin{subequations}
\begin{align}
n_2^3 (L_2L_3 q^2 - L_2L_3 + L_1L_3 n_1^2 - L_1L_2 n_2^2) &= 0 \label{step23-eq1} \\
(q-n_1 + 1) n_2 ( L_2L_3 n_1 (q+1) (q-2) + (n_1 -2) (L_1L_3 n_1^2 - L_1L_2 n_2^2) )&= 0 \label{step23-eq2}
\end{align}
\end{subequations}
Since $0<2 \le n_2,n_1\le q$, from \ref{step23-eq1}, we have $L_1L_3 n_1^2 - L_1L_2 n_2^2 = L_2L_3 (1-q^2)$. Then, after substitution into  \ref{step23-eq2} and
since~$n_2(q-n_1 + 1)\ge n_2>0$, we arrive to the contradiction~$0=L_2L_3 (q+1) (2q -n_1 -2) >L_2L_3 (q+1) (q-2)>0$, because $q\ge4$ and~$L_2,L_3>0$. 
Therefore by Lemma~\ref{lemma-lin-indp2} it follows that~$Z^{(q+1,n_1,n_2)}_{z_{\delta^1}} \in \GG^{q}$. 

By induction, using~\eqref{step361.res1}, \eqref{step361.res2},
\eqref{step362.res2}, \eqref{step362.res3}, and the induction hypothesis~\eqref{IH-Cylqn1n2}, we obtain
\begin{subequations}\label{lem36.res}
\begin{align} 
Z^{j(n),n} \in \GG^{q}  \quad \text{for all~} n  \in \RR^{q+1}_1,
\intertext{and a similar argument gives us}
Z^{j(n),n} \in \GG^{q}  \quad \text{for all~} n  \in \RR^{q+1}_2.
\end{align} 
\end{subequations} \qed

\subsubsection{Proof of Lemma \ref{lem-gen00}} \label{subsec-proof.gen00}

To generate $n = (q+1,0,0)$, we choose
\begin{align*}
k &=(q, 0,1),& m &=(1,0,1),\\ w^k &= (L_1, 0, -L_3q),& w^m &= (L_1,0,-L_3),
\end{align*}
which gives us
\begin{align*}
\left( Y^k \cdot \nabla   \right) Z^m + \left( Z^m \cdot \nabla   \right) Y^k
= \ZZ^{(q+1,0,0)}_{z_{\alpha^1}} + \ZZ^{(q+1,0,2)}_{z_{\alpha^2}} + \ZZ^{(q-1,0,0)}_{z_{\alpha^3}} +  \ZZ^{(q-1,0,2)}_{z_{\alpha^3}}.
\end{align*}
From~\eqref{IH.C} and~\eqref{lem36.res}, we can conclude that $\Pi \ZZ^{(q+1,0,0)}_{z_{\alpha^1}} \in \GG^{q}$ where 
\begin{align*}
z_{\alpha^1}= -\frac{\pi }{4}
\begin{pmatrix}
0 \\
0 \\
 L_3 (q-1)^2 \\
\end{pmatrix}.
\end{align*}
Now, since $L_3 (q-1)^2\ne0$ we have that $Z^{1,(q+1,0,0)} =\zeta \ZZ^{(q+1,0,0)}_{z_{\alpha^1}}$. Therefore 
\begin{subequations}\label{lem37.res}
\begin{align} 
&Z^{1,(q+1,0,0)} \in \GG^{q},
\intertext{and a similar argument gives us}
&Z^{1,(0,q+1,0)} \in \GG^{q}.
\end{align} 
\end{subequations} \qed

\subsubsection{Proof of Lemma \ref{lem-genL}} \label{subsec-proof.genL}
Due to two different types of boundary conditions, we divide the proof into two steps
\begin{enumerate}
\renewcommand{\theenumi}{{\sf3.8.\arabic{enumi}}} 
 \renewcommand{\labelenumi}{} 
\item $\bullet$~Step~\theenumi:\label{Step381} Generating $Z^{j(n),n}$ with $n  \in  \LL^{q+1}_{2,3}\bigcup \LL^{q+1}_{3,1}$.
\item $\bullet$~Step~\theenumi:\label{Step382} Generating $Z^{j(n),n}$ with $n  \in  \LL^{q+1}_{1,2}$.
\end{enumerate}

\medskip
\textbf{Step \ref{Step381}:} Generating $Z^{j(n),n}$ with $n  \in  \LL^{q+1}_{2,3}\bigcup \LL^{q+1}_{3,1}$.

To generate~$n = (l,\,q+1,\,q+1)$, with~$1 \le l \le q$. We start computing~$\left( Y^k \cdot \nabla   \right) Z^m + \left( Z^m \cdot \nabla   \right) Y^k$ with
\begin{align*}
k &=(l, q-1,q),& m &=(0,2,1),\\ w^k &= (0, L_2q, L_3(1-q)),& w^m &= (0,L_2,-2L_3),
\end{align*}
and obtain that $\Pi \ZZ^{(l,q+1,q+1)}_{z_{\alpha^1}} \in \GG^{q}$ with
\begin{align*}
z_{\alpha^1} =  -\frac{\pi}{4} \begin{pmatrix}
0 \\ L_2 (q^2-1) \\  L_3(q+1) (q-3)
\end{pmatrix}.
\end{align*}
Next we compute~$\left( Y^k \cdot \nabla   \right) Z^m + \left( Z^m \cdot \nabla   \right) Y^k$
\begin{align*}
k &=(l, q,q),& m &=(0,1,1),\\ w^k &= (L_1q, -L_2l, 0),& w^m &= (0,L_2,-L_3),
\end{align*}
and obtain $\Pi\ZZ^{(l,q+1,q+1)}_{z_{\gamma^1}}\in \GG^{q}$ with 
\begin{align*}
z_{\gamma^1} = -\frac{\pi}{4} \begin{pmatrix}
0 \\ L_2l \\  L_3l
\end{pmatrix}.
\end{align*}
Since 
\begin{align*} \frac{16}{\pi^2} \det \left(n \mid z_{\alpha^1} \mid z_{\gamma^1} \right)=
\det \begin{pmatrix}
l & 0 & 0 \\
q+1 & L_2(q^2-1) & L_2l \\
q+1 & L_3(q+1)(q-3) & L_3l 
\end{pmatrix} = 2 L_2L_3 l^2 (q+1) > 0,
\end{align*}
from Lemma~\ref{lemma-lin-indp_Z} we obtain that
\begin{subequations}\label{step381.res1}
\begin{align}
  Z^{\{1,2\}, (l,q+1,q+1)}\in \GG^{q},\quad\mbox{for}\quad1 \le l \le q,
  \intertext{and a similar argument gives us}
  Z^{\{1,2\}, (q+1,l,q+1)}\in \GG^{q},\quad\mbox{for}\quad1 \le l \le q,
\end{align}
\end{subequations}

{\em The case~$n= (0,q+1,q+1)$}. We can use again~\cite[Theorem~4.1]{PhanRod-ecc15} to conclude that

\begin{subequations}\label{step381.res2}
\begin{align} 
&Z^{1,(0,q+1,q+1)} \in \GG^{q},
\intertext{and a similarly}
&Z^{1,(q+1,0,q+1)} \in \GG^{q}.
\end{align} 
\end{subequations}

\medskip
\textbf{Step \ref{Step382}:} Generating $Z^{j(n),n}$ with $n  \in  \LL^{q+1}_{1,2}$.

 To generate $n = (q+1, q+1, l)$ with $1 \le l \le q$, we first compute~$\left( Y^m \cdot \nabla   \right) Z^k + \left( Z^k \cdot \nabla   \right) Y^m$ with
\begin{align*}
k &=(q, q,l),& m &=(1,1,0),\\ w^k &= (L_1l, 0, -L_3q),& w^m &= (L_1,-L_2,0),
\end{align*}
which gives us~$\Pi \ZZ^{(q+1,q+1,l)}_{z_{\alpha^1}} \in \GG^{q}$, with 
\begin{align*}
z_{\alpha^1} = \frac{\pi}{4} \begin{pmatrix}
 L_1l \\ - L_2l \\ 0 
\end{pmatrix}. 
\end{align*} 
Next, we choose compute~$\left( Y^m \cdot \nabla   \right) Z^k + \left( Z^k \cdot \nabla   \right) Y^m$ with
\begin{align*}
k &=(q, q-1,l),& m &=(1,2,0),\\ w^k &= (L_1(q-1), -L_2q, 0),& w^m &= (2L_1,-L_2,0),
\end{align*}
which gives us~$\Pi \ZZ^{(q+1,q+1,l)}_{z_{\gamma^1}} \in \GG^{q}$, with 
\begin{align*}
z_{\gamma^1} = \frac{\pi}{4}  \begin{pmatrix}
L_1(q+1)(q-3) \\ -L_2 (q^2-1) \\ 0 
\end{pmatrix}. 
\end{align*} 
From 
\begin{align*}  \frac{16}{\pi^2} \det \left(n \mid z_{\alpha^1}\mid z_{\gamma^1} \right)=
\det \begin{pmatrix}
q+1 &L_1l & L_1(q+1)(q-3) \\
q+1 & -L_2l & -L_2 (q^2-1) \\
l & 0 & 0 
\end{pmatrix} = \frac{32}{\pi^2}L_1L_2 l^2 (q+1) > 0,
\end{align*}
and Lemma~\ref{lemma-lin-indp_Z} we obtain that
\begin{equation}\label{step382.res1}
 Z^{\{1,2\}, (q+1,q+1,l)}\in \GG^{q}, \quad\mbox{for}\quad 1 \le l \le q.
\end{equation}

To generate $n = (q+1,q+1,0)$ we choose
\begin{align*}
k &=(q, q-1,1),& m &=(1,2,1),\\ w^k &= (L_1, 0, -L_3q),& w^m &= (L_1,0,-L_3),
\end{align*}
which gives us
\[
 \left( Y^k \cdot \nabla   \right) Z^m + \left( Z^m \cdot \nabla   \right) Y^k=\ZZ^{(q+1,q+1,0)}_{z_{\alpha^1}}+\ZZ^{(q+1,q+1,2)}_{z_{\alpha^2}}+\sum_{i=3}^8\ZZ^{k^i}_{z_{\alpha^i}}
\]
with~$k^i\in\{(q+1,q-3,0),(q+1,q-3,2),(q-1,q+1,0),(q-1,q+1,2),(q-1,q-3,0),(q-1,q-3,2)\}$. Recalling that~$q\ge4$, from~\eqref{IH.C} and~\eqref{lem36.res} it follows that
$\sum_{i=3}^8\Pi\ZZ^{k^i}_{z_{\alpha^i}}\in \GG^{q}$ and, from~\eqref{step382.res1} we have that~$\Pi\ZZ^{(q+1,q+1,2)}_{z_{\alpha^2}}\in\GG^{q}$. Therefore we obtain
$\Pi \ZZ^{(q+1,q+1,0)}_{z_{\alpha^1}} \in \GG^{q}$, and we can find 
\begin{align*}
z_{\alpha^1} = \frac{\pi}{8}  \begin{pmatrix}
-L_1(q+1) \\ 0 \\ (q+1)(-L_3q+L_3) 
\end{pmatrix}. 
\end{align*} 
Observe that~$\ZZ^{1,(q+1,q+1,0)}_{z_{\alpha^1}}=-\frac{\pi}{8}(q^2-1)L_3\begin{pmatrix}
                                                                    0\\0\\\cos(\frac{(q+1)\pi x_1}{L_1})\cos(\frac{(q+1)\pi x_2}{L_2}) 
                                                                   \end{pmatrix}$, that is,
for a suitable constant~$\zeta\ne0$, we have $Z^{1,(q+1,q+1,0)}=\zeta\ZZ^{1,(q+1,q+1,0)}_{z_{\alpha^1}}$.
In particular~$\Pi\ZZ^{1,(q+1,q+1,0)}_{z_{\alpha^1}}=\ZZ^{1,(q+1,q+1,0)}_{z_{\alpha^1}}$ and it follows
\begin{align} \label{step382.res2}
Z^{1,(q+1,q+1,0)} \in \GG^{q}.
\end{align}

To sum up, from~\eqref{step381.res1}, \eqref{step381.res2}, \eqref{step382.res1}, and~\eqref{step382.res2}, it follows that
\begin{align}\label{lem38.res}
Z^{j(n),n} \in \GG^{q},\quad\mbox{for all}\quad n\in\LL^{q+1}_{1,2}{\textstyle\bigcup}\LL^{q+1}_{2,3}{\textstyle\bigcup}\LL^{q+1}_{3,1}.
\end{align} \qed

\subsubsection{Proof of Lemma \ref{lem-genq+1}}\label{subsec-proof.genq+1}

Firstly, we compute~$\left( Y^k \cdot \nabla   \right) Z^m + \left( Z^m \cdot \nabla   \right) Y^k$
\begin{align*}
k &=(q,q-1,q),& m &=(1,2,1),\\ w^k &=  (L_1,0,-L_3),& w^m &= (2L_3,-L_2,0),
\end{align*}
which gives us~$\Pi \ZZ^{(q+1,q+1,q+1)}_{z_{\alpha^1}}\in \GG^{q}$ where 
\begin{align*} z_{\alpha^1} =  \frac{\pi}{8}
\begin{pmatrix}
2 L_1 (q+1)\\
- L_2(q+1) \\
0
\end{pmatrix}. 
\end{align*}
Secondly, we compute~$\left( Y^m \cdot \nabla   \right) Z^k + \left( Z^k \cdot \nabla   \right) Y^m$ 
\begin{align*}
k &=(q,q-1,q),& m &=(1,2,1),\\ w^k &=  (L_1,0,-L_3),& w^m &= (2L_3,-L_2,0),
\end{align*}
which gives us~$\Pi \ZZ^{(q+1,q+1,q+1)}_{z_{\gamma^1}}\in \GG^{q}$, with
\begin{align*} z_{\gamma^1} = \frac{\pi}{8}
\begin{pmatrix}
- L_1 (q+1)(q-3)\\
 L_2(q^2-1) \\
0
\end{pmatrix}. 
\end{align*}
Since 
\begin{align*} \frac{64}{\pi^2(q+1)^3} \det \left(n \mid z_{\alpha^1}\mid z_{\gamma^1} \right) =
\det \begin{pmatrix}
1 & 2L_1 & -L_1  (q-3) \\ 
1 & -L_2  & L_2 (q-1) \\
1 & 0 & 0 
\end{pmatrix} = L_1L_2(q+1)  > 0, 
\end{align*}
by Lemma~\ref{lemma-lin-indp_Z}, we find
\begin{align} \label{lem39.res} 
Z^{\{1,2\}, (q+1,q+1,q+1)}\in \GG^{q}.
\end{align} \qed

\section{Final remarks}
Following the approximate controllability by degenerate low modes forcing proven in \cite{PhanRod18JDCS,Shirikyan06}, we present an explicit ${(\rm L, D}(A))$-saturating set in a general 3D Cylinder. This case is as an extended result in the work of 2D Cylinder (see \cite{PhanRod-ecc15}). However we just get the control $\eta \in L^\infty((0,T), \GG^1)$ instead of  $L^\infty((0,T), \GG^0)$ in 2D Cylinder case. The reason is that we do not have the equality $B(Y^k, Y^k) = 0$ for all  $k$ as in 2D Cylinder case (see more details in \cite[Theorem 3.2]{PhanRod18JDCS}). 

We underline that the presented saturating set is (by definition) independent of the viscosity coefficient $\nu$. That is, approximate controllability holds by means of controls taking values in $\GG^1 = {\rm span} \CC + {\rm span} \BB(\CC, {\rm span} \CC) =  {\rm span}\left(\CC \cup \BB(\CC, \CC) \right)$, for any $\nu > 0$. Note that a ${(\rm L, D}(A))$-saturating set with  less elements does exist. One of them will be introduced in next corollary. 
\begin{cor} \label{cor-sat-cyl}
The set of eigenfuntions~
\begin{align*}
\widetilde{\CC} \coloneqq &\left\{ Y^{j(n),n} \mid~n\in\N^3,~\#_0(n)\le1,~n_i\le 3,~j(n) \in \{ 1, 2-\#_0(n) \}\right\} \\
& {\textstyle\bigcup}  \left\{ Z^{j(n),n} \mid~n\in\N^3,~\#_0(n)\le1,~n_i\le {4},~j(n) \in \{ 1, 2-\#_0(n) \}\right\} \\
& {\textstyle\bigcup} \left\{ Z^{j(n),n} \mid~n \in \N^3,~\#_0(n)=2,~n_3=0 \right\}  ~ {\textstyle\bigcup}  \left\{Z^{(0,0,0)} \right \}
\end{align*}
is saturating. 
\end{cor}
\begin{proof}
Denote $\widetilde{\CC}  = \CC_Y \cup \widetilde{\CC}_Z$ where 
\begin{align*}
\CC_Y &\coloneqq \left\{ Y^{j(n),n} \mid~n\in\N^3,~\#_0(n)\le1,~n_i\le 3,~j(n) \in \{ 1, 2-\#_0(n) \}\right\}, \\
\widetilde{\CC}_Z &\coloneqq \left\{ Z^{j(n),n} \mid~n\in\N^3,~\#_0(n)\le1,~n_i\le {4},~j(n) \in \{ 1, 2-\#_0(n) \}\right\} \\ 
&\quad  {\textstyle\bigcup}  \left\{ Z^{j(n),n} \mid~n \in \N^3,~\#_0(n)=2,~n_3=0 \right\} ~ {\textstyle\bigcup}  \left\{Z^{(0,0,0)} \right \}. 
\end{align*}
Notice that $\CC_Y$ is the same set as in Theorem \ref{T:satur3Drect}.
We recall the definition of $\CC^q_Y$ and $\CC^q_Z$ as in \eqref{splitCYZ}. 
Using Theorem \ref{T:satur3Drect}, we get that $\CC^q_Y \subseteq \GG^{q-2} \subseteq \GG^{q-1}$ for all $q \in \N_4$. Repeating the arguments in the proof of Theorem \ref{T:satur3Dcyl}, we get that $\CC^q_Z \subseteq \GG^{q-1}$ for all $q \in \N_4$. 
In conclusion it yields that $\widetilde{\CC}$ is also a saturating set with less elements than $\CC$ in Theorem \ref{T:satur3Dcyl}.  
\end{proof}
As mentioned in the beginning of this work, it is not our goal to find a saturating set with minimal number of elements.

\medskip

\end{document}